\documentclass[10pt,oneside,a4paper,draft]{amsart}
\usepackage{xcolor}
\usepackage{tikz}
\usepackage{amsmath,amsfonts,amsthm, amssymb}
\usepackage{array}
\usepackage[all,textures]{xy}
\usepackage{portland}
\usepackage{verbatim}
\numberwithin{equation}{section}
\numberwithin{figure}{section}

\theoremstyle{plain}
\newtheorem{theorem}{Theorem}[section]
\newtheorem{proposition}[theorem]{Proposition}
\newtheorem{lemma}[theorem]{Lemma}
\newtheorem{corollary}[theorem]{Corollary}

\newtheorem{convention}[theorem]{Convention}

\theoremstyle{definition}
\newtheorem{definition}[theorem]{Definition}

\theoremstyle{remark}

\newtheorem{remark}[theorem]{Remark}
\newtheorem{example}[theorem]{Example}

\let\cal\mathcal
\def\Ascr{{\cal A}}
\def\Bscr{{\cal B}}

\def\Escr{{\cal E}}

\def\Mscr{{\cal M}}

\def\Oscr{{\cal O}}

\def\Sscr{{\cal S}}

\def\Xscr{{\cal X}}

\def\id{\text{id}}
\def\Id{\operatorname{id}}

\def\Bimod{\operatorname{Bimod}}

\def\Mod{\operatorname{Mod}}

\def\Ext{\operatorname {Ext}}
\def\Hom{\operatorname {Hom}}

\def\uEnd{\operatorname {\mathcal{E}\mathit{nd}}}
\def\End{\operatorname {End}}

\def\End{\operatorname {End}}

\def\id{{\operatorname {id}}}

\def\r{\rightarrow}

\newcommand{\cone}{\mathrm{cone}}

\renewcommand{\lim}{\mathrm{lim}}

\newcommand{\Ob}{\mathrm{Ob}}
\newcommand{\Z}{\mathbb{Z}}

\newcommand{\N}{\mathbb{N}}

\newcommand{\B}{\mathbb{B}}

\newcommand{\CC}{\mathbf{C}}

\newcommand{\lra}{\rightarrow}

\newcommand{\aaa}{\ensuremath{\mathcal{A}}}

\newcommand{\ccc}{\ensuremath{\mathcal{C}}}

\newcommand{\eee}{\ensuremath{\mathcal{E}}}

\newcommand{\mmm}{\ensuremath{\mathcal{M}}}

\newcommand{\ooo}{\ensuremath{\mathcal{O}}}

\newcommand{\sss}{\ensuremath{\mathcal{S}}}

\newcommand{\xxx}{\ensuremath{\mathcal{X}}}

\CompileMatrices
\SelectTips{cm}{11}

\let\oldmarginpar\marginpar
\long\def\marginpar#1{\oldmarginpar{\raggedright \tiny \baselineskip 0pt \lineskip 0pt #1}}

\def\Bar{\operatorname{Bar}}
\def\Hoch{\operatorname{Hoch}}
\def\coHoch{\operatorname{coHoch}}
\def\scHoch{\operatorname{scHoch}}

\def\REnd{\operatorname{REnd}}
\def\Ho{\operatorname{Ho}}
\def\dgAlg{\operatorname{dgAlg}}
\def\elemi{%
\begin{tikzpicture}[scale=0.7,baseline=(current bounding box.center)]
\path[draw,fill=lightgray!20!] (2.5,2.0) rectangle (3.5,3.0);
\path[draw,fill=lightgray] (0.0,3.0) rectangle (4.0,4.0);
\path[draw,fill=lightgray] (2.0,1.0) rectangle (6.0,2.0);
\path[draw] (0.5,0.0) -- (0.5,3.0);
\path[draw] (1.0,4.0) -- (1.0,5.0);
\path[draw] (1.5,0.0) -- (1.5,3.0);
\path[draw] (3.0,0.0) -- (3.0,1.0);
\path[draw] (2.5,2.0) -- (2.5,3.0);
\path[draw] (3.5,2.0) -- (3.5,3.0);
\path[draw] (3.0,4.0) -- (3.0,5.0);
\path[draw] (4.5,2.0) -- (4.5,5.0);
\path[draw] (5.0,0.0) -- (5.0,1.0);
\path[draw] (5.5,2.0) -- (5.5,5.0);
\node[anchor=center] at (2.0,3.5) {$\psi$};
\node[anchor=center] at (4.0,1.5) {$\phi$};
\node[anchor=center] at (1.0,1.5) {$\cdots$};
\node[anchor=center] at (5.0,3.5) {$\cdots$};
\node[anchor=center] at (3.0,2.5) {$\cdots$};
\node[anchor=center] at (2.0,4.5) {$\cdots$};
\node[anchor=center] at (4.0,0.5) {$\cdots$};
\path[draw,fill=black] (0.5,3.0) circle [radius=0.05];
\path[draw,fill=black] (1.5,3.0) circle [radius=0.05];
\path[draw,fill=black] (2.5,3.0) circle [radius=0.05];
\path[draw,fill=black] (3.5,3.0) circle [radius=0.05];
\path[draw,fill=black] (2.5,2.0) circle [radius=0.05];
\path[draw,fill=black] (3.5,2.0) circle [radius=0.05];
\path[draw,fill=black] (4.5,2.0) circle [radius=0.05];
\path[draw,fill=black] (5.5,2.0) circle [radius=0.05];
\path[draw,fill=black] (1.0,4.0) circle [radius=0.05];
\path[draw,fill=black] (3.0,4.0) circle [radius=0.05];
\path[draw,fill=black] (3.0,1.0) circle [radius=0.05];
\path[draw,fill=black] (5.0,1.0) circle [radius=0.05];
\end{tikzpicture}%
}%
\def\elemii{%
\begin{tikzpicture}[scale=0.7,baseline=(current bounding box.center)]
\path[draw,fill=lightgray!20!] (2.5,2.0) rectangle (3.5,3.0);
\path[draw,fill=lightgray] (2.0,3.0) rectangle (6.0,4.0);
\path[draw,fill=lightgray] (0.0,1.0) rectangle (4.0,2.0);
\path[draw] (5.5,0.0) -- (5.5,3.0);
\path[draw] (5.0,4.0) -- (5.0,5.0);
\path[draw] (4.5,0.0) -- (4.5,3.0);
\path[draw] (3.0,0.0) -- (3.0,1.0);
\path[draw] (3.5,2.0) -- (3.5,3.0);
\path[draw] (2.5,2.0) -- (2.5,3.0);
\path[draw] (3.0,4.0) -- (3.0,5.0);
\path[draw] (1.0,0.0) -- (1.0,1.0);
\path[draw] (0.5,2.0) -- (0.5,5.0);
\path[draw] (1.5,2.0) -- (1.5,5.0);
\node[anchor=center] at (4.0,3.5) {$\psi$};
\node[anchor=center] at (2.0,1.5) {$\phi$};
\node[anchor=center] at (1.0,3.5) {$\cdots$};
\node[anchor=center] at (5.0,1.5) {$\cdots$};
\node[anchor=center] at (3.0,2.5) {$\cdots$};
\node[anchor=center] at (2.0,0.5) {$\cdots$};
\node[anchor=center] at (4.0,4.5) {$\cdots$};
\path[draw,fill=black] (4.5,3.0) circle [radius=0.05];
\path[draw,fill=black] (5.5,3.0) circle [radius=0.05];
\path[draw,fill=black] (2.5,3.0) circle [radius=0.05];
\path[draw,fill=black] (3.5,3.0) circle [radius=0.05];
\path[draw,fill=black] (2.5,2.0) circle [radius=0.05];
\path[draw,fill=black] (3.5,2.0) circle [radius=0.05];
\path[draw,fill=black] (0.5,2.0) circle [radius=0.05];
\path[draw,fill=black] (1.5,2.0) circle [radius=0.05];
\path[draw,fill=black] (5.0,4.0) circle [radius=0.05];
\path[draw,fill=black] (3.0,4.0) circle [radius=0.05];
\path[draw,fill=black] (3.0,1.0) circle [radius=0.05];
\path[draw,fill=black] (1.0,1.0) circle [radius=0.05];
\end{tikzpicture}%
}%
\def\elemiii{%
\begin{tikzpicture}[scale=0.7,baseline=(current bounding box.center)]
\path[draw,fill=lightgray!20!] (2.5,2.0) rectangle (3.5,3.0);
\path[draw,fill=lightgray] (0.0,3.0) rectangle (6.0,4.0);
\path[draw,fill=lightgray] (2.0,1.0) rectangle (4.0,2.0);
\path[draw] (0.5,0.0) -- (0.5,3.0);
\path[draw] (1.5,0.0) -- (1.5,3.0);
\path[draw] (4.5,0.0) -- (4.5,3.0);
\path[draw] (5.5,0.0) -- (5.5,3.0);
\path[draw] (2.5,0.0) -- (2.5,1.0);
\path[draw] (3.5,0.0) -- (3.5,1.0);
\path[draw] (2.5,2.0) -- (2.5,3.0);
\path[draw] (3.5,2.0) -- (3.5,3.0);
\path[draw] (1.0,4.0) -- (1.0,5.0);
\path[draw] (5.0,4.0) -- (5.0,5.0);
\node[anchor=center] at (3.0,3.5) {$\psi$};
\node[anchor=center] at (3.0,1.5) {$\phi$};
\node[anchor=center] at (1.0,1.5) {$\cdots$};
\node[anchor=center] at (5.0,1.5) {$\cdots$};
\node[anchor=center] at (3.0,0.5) {$\cdots$};
\node[anchor=center] at (3.0,2.5) {$\cdots$};
\node[anchor=center] at (2.0,4.5) {$\cdots$};
\node[anchor=center] at (4.0,4.5) {$\cdots$};
\path[draw,fill=black] (4.5,3.0) circle [radius=0.05];
\path[draw,fill=black] (5.5,3.0) circle [radius=0.05];
\path[draw,fill=black] (2.5,3.0) circle [radius=0.05];
\path[draw,fill=black] (3.5,3.0) circle [radius=0.05];
\path[draw,fill=black] (2.5,2.0) circle [radius=0.05];
\path[draw,fill=black] (3.5,2.0) circle [radius=0.05];
\path[draw,fill=black] (0.5,3.0) circle [radius=0.05];
\path[draw,fill=black] (1.5,3.0) circle [radius=0.05];
\path[draw,fill=black] (5.0,4.0) circle [radius=0.05];
\path[draw,fill=black] (1.0,4.0) circle [radius=0.05];
\path[draw,fill=black] (2.5,1.0) circle [radius=0.05];
\path[draw,fill=black] (3.5,1.0) circle [radius=0.05];
\end{tikzpicture}%
}%
\def\elemiv{%
\begin{tikzpicture}[scale=0.7,baseline=(current bounding box.center)]
\path[draw,fill=lightgray!20!] (2.5,2.0) rectangle (3.5,3.0);
\path[draw,fill=lightgray] (0.0,1.0) rectangle (6.0,2.0);
\path[draw,fill=lightgray] (2.0,3.0) rectangle (4.0,4.0);
\path[draw] (0.5,2.0) -- (0.5,5.0);
\path[draw] (1.5,2.0) -- (1.5,5.0);
\path[draw] (4.5,2.0) -- (4.5,5.0);
\path[draw] (5.5,2.0) -- (5.5,5.0);
\path[draw] (2.5,4.0) -- (2.5,5.0);
\path[draw] (3.5,4.0) -- (3.5,5.0);
\path[draw] (2.5,2.0) -- (2.5,3.0);
\path[draw] (3.5,2.0) -- (3.5,3.0);
\path[draw] (1.0,0.0) -- (1.0,1.0);
\path[draw] (5.0,0.0) -- (5.0,1.0);
\node[anchor=center] at (3.0,3.5) {$\psi$};
\node[anchor=center] at (3.0,1.5) {$\phi$};
\node[anchor=center] at (1.0,3.5) {$\cdots$};
\node[anchor=center] at (5.0,3.5) {$\cdots$};
\node[anchor=center] at (3.0,4.5) {$\cdots$};
\node[anchor=center] at (3.0,2.5) {$\cdots$};
\node[anchor=center] at (2.0,0.5) {$\cdots$};
\node[anchor=center] at (4.0,0.5) {$\cdots$};
\path[draw,fill=black] (4.5,2.0) circle [radius=0.05];
\path[draw,fill=black] (5.5,2.0) circle [radius=0.05];
\path[draw,fill=black] (2.5,3.0) circle [radius=0.05];
\path[draw,fill=black] (3.5,3.0) circle [radius=0.05];
\path[draw,fill=black] (2.5,2.0) circle [radius=0.05];
\path[draw,fill=black] (3.5,2.0) circle [radius=0.05];
\path[draw,fill=black] (0.5,2.0) circle [radius=0.05];
\path[draw,fill=black] (1.5,2.0) circle [radius=0.05];
\path[draw,fill=black] (2.5,4.0) circle [radius=0.05];
\path[draw,fill=black] (3.5,4.0) circle [radius=0.05];
\path[draw,fill=black] (1.0,1.0) circle [radius=0.05];
\path[draw,fill=black] (5.0,1.0) circle [radius=0.05];
\end{tikzpicture}%
}%

\def\Bar{B}
\let\longmapsto\mapsto

\title[$B_\infty$-structure]{The \boldmath{$B_\infty$}-structure on the derived endomorphism algebra of the unit in a monoidal category}

\author{Wendy Lowen} 
\address[Wendy Lowen]{Universiteit Antwerpen, Departement Wiskunde-Informatica, Middelheimcampus,
Middelheimlaan 1,
2020 Antwerp, Belgium}
\address{Laboratory of Algebraic Geometry, National Research University, Higher School of Economics, Moscow, Russia}
\email{wendy.lowen@uantwerpen.be}

\author{Michel Van den Bergh}
\address[Michel Van den Bergh]{Universiteit Hasselt\\ Campus Diepenbeek\\ Agoralaan Gebouw D \\ 3590 Diepenbeek \\Belgium}
\email{michel.vandenbergh@uhasselt.be}

\thanks{The authors acknowledge the support of the Research Foundation Flanders (FWO) under Grant No G.0D86.16N, of the European Union for the ERC grant No 817762 - FHiCuNCAG and of the Russian Academic Excellence Project `5-100'.}


\begin{document}
\maketitle

\begin{abstract}
Consider a monoidal category which is at the same time abelian with enough projectives and such that projectives are flat on the right. We show that there is a $B_{\infty}$-algebra which is $A_{\infty}$-quasi-isomorphic to the derived endomorphism algebra of the tensor unit. This $B_{\infty}$-algebra is obtained as the co-Hochschild complex of a projective resolution of the tensor unit, endowed with a lifted $A_{\infty}$-coalgebra structure. We show that in the classical situation of the category of bimodules over an algebra, this newly defined $B_{\infty}$-algebra is isomorphic to the Hochschild complex of the algebra in the homotopy category of $B_{\infty}$-algebras.
\end{abstract}

\section{Introduction}
Throughout $k$ is a commutative base ring which we will often not mention.
Ever since the pioneering work of Gerstenhaber in the 1960s on the deformation theory
of algebras, the algebraic structure of the Hochschild cohomology
$\operatorname{HH}^{\ast}(A)=\Ext^\ast_{\Bimod(A)}(A,A)$ of an associative algebra $A$ has been a topic of
considerable interest, with the structure of the Hochschild complex $\CC(A)$,
computing this cohomology, entering the picture somewhat later. Whereas
the former becomes a Gerstenhaber algebra when equipped with the cup product and the Gerstenhaber bracket, the latter becomes a
$B_{\infty}$-algebra when endowed with the Hochschild differential, the
cup product and the  brace operations \cite{getzlerjones}. This $B_{\infty}$-structure is a stepping stone
in the construction of an algebra structure on $\CC(A)$ over the chain little disk operad,
a result famously known as Deligne's conjecture, for which
various proofs and generalisations are currently available
\cite{francis,kontsevichsoibelmandeligne,lurieha,mccluresmith,shoikhet1,shoikhet2}.

\medskip

The question whether it is possible in the above results to replace
the category of~$A$-bimodules by more general monoidal categories is a
natural one.  A result on the cohomology level has been obtained in
\cite{hermann} in the context of suitable exact monoidal
categories $(\Ascr,\otimes,I)$. In loc.\ cit.\ Hermann exhibits the existence of a 
``bracket'' on $\Ext^\ast_{\Ascr}(I,I)$ (a binary operation whose desirable properties in general have not actually been established), generalising earlier work by
Schwede \cite{schwedealg} in the algebra case. The bracket is obtained
as a lifting of the Eckmann-Hilton argument for the commutative algebra
$\Ext^\ast_{\Ascr}(I,I)$ to the corresponding extension
categories $\Escr\mathit{xt}_{\Ascr}^\ast(I,I)$. See~\cite{NeemanRetakh}.

\medskip

As a complex (and in fact as a a dg-algebra) $\CC(A)$ computes the derived endomorphism algebra $\REnd_{\Bimod(A)}(A)$ of $A$ as a bimodule and, inspired by this, \cite{shoikhet1,shoikhet2}, 
Shoikhet, using very different methods, generalised the presence of a
$B_{\infty}$-structure to a suitable model for $\REnd_{\Ascr}(I)$ thereby to some extent also recovering and generalising Hermann's result.

\medskip

The main purpose of this paper is to exhibit a different and rather straightforward approach to Shoikhet's results.
Our main result is the following:

\begin{theorem}[see Theorem \ref{th:mainth}] \label{th:mainth:intro}
Let $(\aaa, \otimes , I)$ be a monoidal category which is at the same time abelian.
Assume that $\Ascr$ has enough projectives and that $P\otimes-$ is exact for $P$ projective. Let $\bold{A}=\REnd_{\Ascr}(I)$. Then there exists a $B_\infty$-algebra $\bold{B}$ together
with an $A_\infty$-quasi-isomorphism $\bold{B}\r \bold{A}$.
\end{theorem}
The conditions on $\aaa$ in Theorem
\ref{th:mainth:intro} can probably be weakened but we
have made no attempt to do so as we mainly want to emphasise the method of proof. The  conditions we impose are in fact of  similar flavour as those in
\cite{hermann,shoikhet1,shoikhet2}, although the precise relationship
remains to be elucidated.

\medskip

Our starting point is the trivial coalgebra structure present on the
tensor unit~$I$, which can be lifted to an
$A_{\infty}$-coalgebra structure on a projective resolution $P$ of~$I$, see Proposition \ref{propcoalgP} (in the case of bimodule
categories this lifting property was independently established in
\cite{negronvolkovwitherspoon}). The proof of Theorem
\ref{th:mainth:intro} is essentially finished by putting
$$\bold{B} = \CC_{\mathrm{coHoch}}(P) = \prod_n\Sigma^{-n}\Hom_{\aaa}(P, P^{\otimes n}),$$
the co-Hochschild complex associated to $P$.
Here, $\bold{B}$ is endowed with a co-Hoch\-schild differential and co-braces among others, and the resulting $B_{\infty}$-structure is analogous, but not quite dual (due to the product in both constructions) to that on the Hochschild complex of an $A_{\infty}$-algebra. In fact both constructions are special instances of a $B_{\infty}$-algebra obtained from a suitable properad (see Proposition \ref{proppropBinfty} and \cite{merkulovvallette}). 

\medskip

Now let $A$ be a unital $k$-algebra and put  $\aaa = \Bimod(A)$. In this case the classical Hochschild complex $\CC_{\Hoch}(A)$ and the co-Hochschild complex $\CC_{\coHoch}(\tilde{B}A)$ for the Bar $A$-bimodule-resolution $\tilde{B}A \r A$ are of a quite different nature with the former having only non-zero braces and the latter having only non-zero co-braces. Nevertheless, we prove the following:

\begin{theorem}[see Theorem \ref{th:mainth}] \label{mainthintro} Assume that $C$ is a counital dg-coalgebra in $C(\Bimod(A))$ such that $\epsilon:C\r A$ is a quasi-isomorphism.  Assume furthermore that~$A$ is $k$-projective and that $C$ is homotopy projective as complex of bimodules.  There is an isomorphism between $\CC_{\Hoch}(A)$ and $\CC_{\coHoch}(C)$ in the homotopy category of $B_\infty$-algebras.
\end{theorem}
Putting $C=\tilde{B}A$ yields:
\begin{corollary}[see Corollary \ref{cor:BA}]\label{maincorintro}
Assume that $A$ is $k$-projective. There is an isomorphism between $\CC_{\Hoch}(A)$ and $\CC_{\coHoch}(\tilde{B}A)$ in the homotopy category of $B_\infty$-algebras, where $\tilde{B}A$
is considered as a coalgebra in $C(\Bimod(A))$.
\end{corollary}

In order to appreciate these results, first note that for $C = \tilde{B}A$ the Hochschild complex is naturally contained inside the co-Hochschild complex as ``first column'', through the cano\-nical isomorphism 
\begin{equation}\label{tildeintro}
\sigma_0: \CC_{\Hoch}(A) \cong \Hom_{A-A}(\tilde{B}A, A): \phi \longmapsto \tilde{\phi}
\end{equation}
which is given by freely extending $\phi \in \Hom_k(A^{\otimes n}, A)$ to $\tilde{\phi} \in \Hom_{A-A}(A \otimes A^{\otimes n} \otimes A, A)$. 
Further, the well known isomorphism
$$\CC_{\Hoch}(A) \cong \mathrm{Coder}_k(BA)$$
from \cite{quillencyclic,stasheffbracket} can be naturally extended to yield
\begin{equation}\label{Hintro}
\sigma_1: \CC_{\Hoch}(A) \lra \Sigma^{-1}\Hom_{A-A}(\tilde{B}A, \tilde{B}A)
\end{equation}
and a computation shows $\sigma = (\sigma_0, \sigma_1)$ to define a morphism of dg Lie algebras $\sigma: \CC_{\Hoch}(A) \lra \CC_{\coHoch}(\tilde{B}A)$, quasi-inverse to the projection onto the first column.
Note that in the recent work \cite{negronvolkovwitherspoon}, the authors prove the existence of an isomorphism $$HH^{\ast}(A) \cong \Sigma \mathrm{Coder}_{A-A}^{\infty}(P)/\mathrm{Inn}_{A-A}^{\infty}(P)$$ with the shifted module of outer $A_{\infty}$-coderivations of an arbitrary bimodule resolution $P$ of $A$. For $P = \tilde{B}A$, the $A_{\infty}$-coderivation corresponding to a Hochschild cocycle $\phi$ has only two components, given precisely by $(\sigma_0(\phi), \sigma_1(\phi))$.

Unfortunately, the dg Lie morphism $\sigma$ does not preserve the dot product so it fails to be a $B_{\infty}$-quasi-isomorphism. Conjecturally, adding on higher components to $\sigma$ will allow the definition of a morphism on the level of the respective Bar constructions. This is work in progress.

\medskip

 In the present paper we develop an altogether different approach to the comparison of $B_{\infty}$-structures, inspired by Keller's arrow category construction for dg categories from \cite{keller}.
 In order to compare Hochschild and co-Hochschild complexes as in Theorem \ref{mainthintro}, we introduce the notion of a semi-coalgebra (Definition \ref{def:semico}) as the ``glueing'' of an algebra $A$ and a coalgebra $C$ along suitable bimodules $M$ and $N$ in both directions. In analogy with the definitions of $A_{\infty}$-algebra and -coalgebra structures, a semi-co-algebra structure is defined as a solution of the Maurer-Cartan equation in $\CC_{\scHoch}(A,M,C,N)$, the \emph{semi-co-Hochschild object}. This $B_{\infty}$-algebra is obtained through Proposition \ref{proppropBinfty} from the \emph{restricted endomorphism properad} described in \S \ref{sec:restricted_end}. In the special case where $A$ is a $k$-algebra, $C$ is a coalgebra in $C(\Bimod(A))$ and $N = 0$, in \S \ref{sec:conditions} we formulate conditions ensuring that the two canonical $B_{\infty}$ projections $\CC_{\scHoch}(A,M,C,N) \r \CC_{\Hoch}(A)$ and $\CC_{\scHoch}(A,M,C,N) \r \CC_{\coHoch}(C)$ are quasi-isomorphisms, and a forteriori that $\CC_{\Hoch}(A) \cong \CC_{\coHoch}(C)$ in $\mathrm{Ho}(B_{\infty})$. Finally, Theorem \ref{mainthintro} is obtained as an application to $M = C$.

\section{Notation and conventions}
\label{sec:notconv}
Throughout $k$ is a commutative ring. Unless otherwise specified objects and categories will be $k$-linear.
We will first discuss our sign conventions.
If $\Ascr$ is a graded category (i.e.\ a category enriched in graded $k$-modules) then a \emph{suspension} of $A\in \Ob(\Ascr)$ is 
a pair $(A',\eta)$ where $A'\in \Ob(\Ascr)$ and $\eta:A\r A'$ is an isomorphism of degree $-1$. The dual notion of \emph{desuspension}
is defined similarly. If $\Ascr$ is a dg-category then we require $\eta$ to be closed as well.

If every object in $\Ascr$ has a suspension $(A',\eta_A)$ then we can define an endofunctor
$\Sigma:\Ascr\r \Ascr$
which on objects takes the value $\Sigma A=A'$ and which on maps $f:A\r B$
is given by\footnote{The sign in this equation is dictated by the fact that if $\Ascr$ is a dg-category then we want $\Sigma$ to be a dg-functor. 
In particular, since $\Sigma$ has degree zero,
we want $d(\Sigma f)=\Sigma df$.}
\begin{equation}
\label{diag:basiccommutativity}
\Sigma f=(-1)^{|f|}\eta_B\circ f\circ \eta_A^{-1}
\end{equation}
We call $\Sigma$ a \emph{shift functor}.\footnote{Note that in general
  the shift functor is not canonical, despite being unique up to
  unique isomorphism. This ambiguity in the choice of a shift functor
  sometimes cause sign complications which maybe avoided by phrasing statements in terms
  of suspensions. In \cite{RizzardoVdB} the axioms
  of triangulated categories were reformulated in such a way that they
  do not refer to a shift functor.} Note that it is easy to formally
close a graded category under suspensions and desuspensions by
adjoining the appropriate morphisms and objects.

If $\Bscr$ is an abelian category then $G(\Bscr)$ and $C(\Bscr)$ denote respectively the 
categories of $\Z$-graded objects $(A_i)_{i\in \Z}$ 
and $\Z$-indexed complexes $(A_i,d_i{:}A_i{\r} A_{i+1})_{i\in \Z}$ over~$\Bscr$.

The categories $G(\Bscr)$, $C(\Bscr)$ are both
equipped with a canonical shift functor which on objects
acts as $(\Sigma A)_i=A_{i+1}$ and (in the case of $C(\Bscr)$) $(\Sigma d)_i=-d_{i+1}$ and for which $\eta_i:A_i\r (\Sigma A)_{i-1}=A_i$ is the identity. If
$f:A\r B$ is a graded morphism of degree $|f|$ then
one checks, using \eqref{diag:basiccommutativity},
that $(\Sigma f)_i$, which is a morphism $A_{i+1}=(\Sigma A)_i\r (\Sigma B )_{i+|f|}=B_{i+|f|+1}$, is
given by $(-1)^{|f|}f_{i+1}$.

If $R$ is a ring then we write $G(R)$, $C(R)$ for $G(\Mod(R))$, $C(\Mod(R))$. In that case
we write $sr:=\eta(r)$.

\section{Properads}\label{sec:properads}
Throughout, $\sss$ will be either the
category $G(k)$ of graded $k$-modules or the category $C(k)$ of
complexes of $k$-modules. We will consider properads in the
symmetric monoidal category $\sss$.  In this section, we collect the
necessary preliminaries on properads and associated
$B_{\infty}$-structures.
\subsection{Generalities}
\label{sec:generalities}
We will assume that the reader is familiar with prop-type structures and in particular
properads (possibly coloured). Nonetheless we give a quick introduction to them, mainly to
fix notations. 

We will first consider the case of an asymmetric properad $\ooo$ in $\sss$.
We denote sequences of numbers by $\underline{m} = (m_1, \dots, m_s)$ and we put $\lambda(\underline{m}) = s$ and $|\underline{m}| = \sum_{i = 1}^s m_i$.
Consider sequences $\underline{m}$, $\underline{n}$, $\underline{k}$, $\underline{l}$ with $\lambda(\underline{m}) = \lambda(\underline{n}) = s$ and $\lambda(\underline{k}) = \lambda(\underline{l}) = t$ and $|\underline{l}| = |\underline{m}|$. In an asymmetric properad we have compositions
\begin{equation}\label{propcomp}
\circ = \circ^{\underline{m}, \underline{n}}_{\underline{k}, \underline{l}}: \bigotimes_{i = 1}^s \ooo(m_i, n_i) \otimes \bigotimes_{j = 1}^t \ooo(k_j, l_j) \lra \ooo(|\underline{k}|, |\underline{n}|)
\end{equation}
whenever the planar graph connecting the $|l|$ outputs to the $|m|$ inputs is connected.
Moreover $\Oscr(1,1)$ contains a specific ``identity'' constant denoted by ``$\Id$''. Using $\Id$ it is easy to see that every 
composition in $\Oscr$ can be rewritten in terms of ``\emph{elementary compositions}'' of the type
\[
(\id,\ldots,\id,\phi,\id,\ldots,\id)\circ (\id,\ldots,\id,\psi,\id,\ldots,\id)
\]
Topological constraints lead to 4 different types of elementary compositions which we give in Figure \ref{fig:watson}.
\begin{figure}[h]
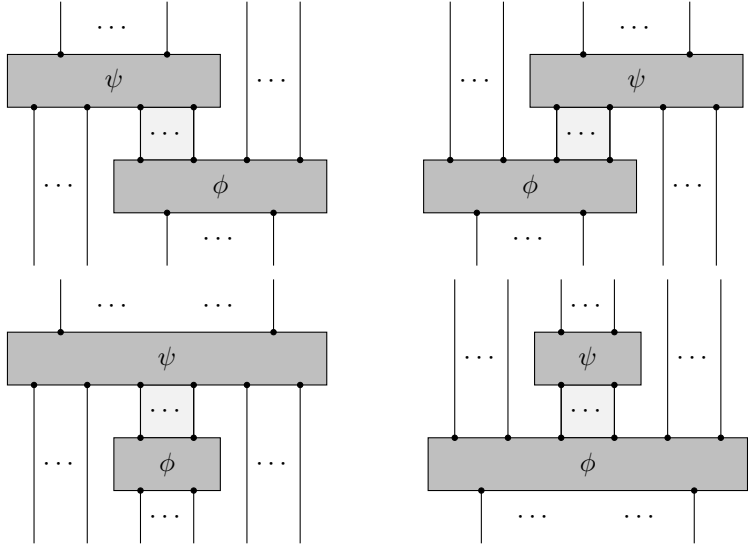

\begin{tabular}{cp{0.5cm}c}
\elemi
&&
\elemii
\\[1.8cm]
\elemiii
&&\
\elemiv
\end{tabular}
\caption{Possible elementary compositions for asymmetric properads.}
\label{fig:watson}
\end{figure}
\begin{example} If $\Oscr$ is an operad (i.e. $\Oscr(m,n)=0$ for $n\neq 1$) then the elementary compositions
are usually denoted by $\circ_i$ where $\phi \circ_i \psi$ means ``insert the evaluation of $\psi$ into $\phi$ as  its $i$'th argument and do not change the other arguments''.
\end{example}
We define the \emph{connection arity} of an elementary compositions as
the number of internal connections (those indicated in light grey in
Figure \ref{fig:watson}). The connectedness requirement for a properad implies that the connection arity is always
$\ge 1$. We say that a properad has \emph{bounded connectivity} if
there exists $p\ge 1$ such that all elementary operations with
connection arity $>p$ are zero.

\medskip

We will also consider \emph{coloured asymmetric properads}. If $\xxx$ is the colour set
we now have $\sss$-objects
$\ooo(x,y)$ for $x = (X_1, \dots, X_k)$ and $y = (Y_1, \dots, Y_n)$ with $X_i, Y_j \in \xxx$. 
The compositions $\circ$ from \eqref{propcomp} have coloured counterparts $\circ_{\underline{u}, \underline{v}}^{\underline{x}, \underline{y}}$ where $\underline{m} = (m_1, \dots, m_s)$ is replaced by $$\underline{x} = ((X^1_1, \dots, X^1_{m_1}), \dots, (X^s_1, \dots, X^s_{m_s}))$$ for colours $X^i_j$ and $Y^i_j$. We put $\lambda(\underline{x}) = s$, $|\underline{x}| = (X^1_1, \dots, X^1_{m_1}, \dots, X^s_1, \dots, X^s_{m_s})$
and $||\underline{x}|| = \sum_{i =1}^s m_i$. There is now an identity $\Id_X\in \Oscr(X,X)$ for every $X\in \Xscr$. Elementary compositions can be defined as in Figure \ref{fig:watson} and the notions
related to it generalise without difficulty to the coloured context.

\begin{convention}
Below we will write ``properad'' for ``asymmetric properad''. 
\end{convention}
\subsection{Shifted properads and signs}
\label{sec:shifted}

If $\Oscr=(\Oscr(m,n))_{m,n}$ is a properad then we frequently need its \emph{shifted versions}
$
\Xi^a \Oscr
$. Those are such that any algebra over $\Xi^a\Oscr$ is of the form $\Sigma^a A$ with $A$ an algebra over $\Oscr$. Checking degrees we find that we should have
\[
(\Xi^a \Oscr)(m,n)=\Sigma^{a(n-m)} (\Oscr(m,n))
\]
To obtain the correct signs in compositions\footnote{The signs are caused by the need to linearise expressions.}  we can write informally
\begin{equation}
\label{eq:shiftedprop}
(\Xi^a \Oscr)(m,n)=(\underbrace{\eta\otimes \cdots \otimes \eta}_n )^a\circ 
\Oscr(m,n)\circ (\underbrace{\eta\otimes \cdots \otimes \eta}_m)^{-a}
\end{equation}
where $\eta$ symbolises suspension (see \S\ref{sec:notconv}) and 
where the $\eta$-expressions can be rewritten using the Koszul sign convention. E.g.
\[
(\eta^{\otimes n})^a
=(-1)^{n(n-1)/2+a(a-1)/2}
(\eta^a)^{\otimes n}\,.
\]
So if $f\in \Oscr(m,n)$ then we have an corresponding operation $\Xi^a f\in (\Sigma^a \Oscr)(m,n)$ of degree $|f|+a(m-n)$
which can be informally written with a formula\footnote{
The sign in \eqref{eq:shiftedprop1} is to ensure that $\Xi:\Oscr(m,n)\r (\Xi\Oscr)(m,n)$ is a morphism of complexes (of degree
$(m-n)$).
It is possible to introduce an additional ``global'' sign in \eqref{eq:shiftedprop1} (i.e.\ depending on $m,n$ but not on $f$) but there is not really any canonical 
choice.
}
 similar to \eqref{diag:basiccommutativity}:
\[
\Xi^a f=\Xi(\Xi\cdots(\Xi f)\cdots)
\]
and
\begin{equation}
\label{eq:shiftedprop1}
\Xi f=(-1)^{m|f|} (\eta^{\otimes n})\circ f\circ (\eta^{\otimes m})^{-1}.
\end{equation}
As an example, one verifies
\[
\Xi^{-1}f=(-1)^{(n+1+|f|)m}(\eta^{\otimes n})^{-1}\circ f\circ (\eta^{\otimes m}).
\]

Note that this works as in any graph defining a composition in $\Xi^a \Oscr$ the inner $\eta$'s will cancel
and we obtain an expression of the same form as \eqref{eq:shiftedprop1}.
\begin{remark}
\label{rem:shifts}
It is possible to make
\eqref{eq:shiftedprop}\eqref{eq:shiftedprop1} mathematically rigorous by introducing colours
(see above). E.g. to obtain $\Xi\Oscr$ we introduce an extra colour
$(-)'$ and we write a sequence of $n$ in/outputs in that colour as $n'$.
We then freely adjoin to $\Oscr$ a two sided invertible
element $\eta\in \Oscr(1,1')_{-1}$.
If the resulting properad is denoted by $\tilde{\Oscr}$
then $(\Xi \Oscr)(m,n)=\tilde{\Oscr}(m',n')$. To obtain $\Xi^2\Oscr$ one introduces 
yet another colour $(-)''$ as well as an invertible element $\eta'\in \Oscr(1',1'')$ of degree $-1$.
The notation $\eta^2$ then stands for $\eta'\circ \eta$. Etc\dots
\end{remark}
\begin{example}\label{ex:assoc}
  Let $\Oscr$ be the associative operad and let $m_2\in \Oscr(2)$ be
  the multiplication. Put
  $b_2:=\Xi m_2=\eta\circ m_2\circ (\eta\otimes \eta)^{-1}=-\eta\circ
  m_2\circ (\eta^{-1}\otimes \eta^{-1})$. Then $|b_2|=1$ and
\begin{align*}
b_2\circ (b_2\otimes \Id)&=\eta\circ m_2 \circ (\eta^{-1}\otimes \eta^{-1}) \circ (\eta\circ m_2\circ(\eta^{-1}\otimes \eta^{-1})\otimes \Id) \\
&=-\eta\circ m_2 \circ (m_2\circ(\eta^{-1}\otimes \eta^{-1})\otimes \eta^{-1}) \\
&=-\eta\circ m_2\circ (m_2\otimes \Id)\circ (\eta^{-1}\otimes \eta^{-1}\otimes \eta^{-1})\\
&=\eta\circ m_2\circ (m_2\otimes \Id)\circ (\eta\otimes \eta\otimes \eta)^{-1}
\end{align*}
On the other hand
\begin{align*}
b_2\circ (\Id\otimes b_2)&=\eta\circ m_2 \circ (\eta^{-1}\otimes \eta^{-1}) \circ (\Id\otimes \eta\circ m_2\circ(\eta^{-1}\otimes \eta^{-1})) \\
&=\eta\circ m_2 \circ (\eta^{-1}\otimes m_2\circ(\eta^{-1}\otimes \eta^{-1})) \\
&=\eta\circ m_2\circ (\Id\otimes m_2)\circ (\eta^{-1}\otimes \eta^{-1}\otimes \eta^{-1})\\
&=-\eta\circ m_2\circ (\Id \otimes m_2)\circ (\eta\otimes \eta\otimes \eta)^{-1}
\end{align*}
so that we obtain in particular the well-known identity
\[
b_2\circ (b_2\otimes \Id)+b_2\circ(\Id\otimes b_2)=0
\]
\end{example}
The above sign conventions extend without difficulty to the case where
$\Oscr$ is a coloured properad. Let $\Xscr$ be the colour set
and assume we  have a weight function
$$a: \xxx \lra \Z: X \longmapsto a_X.$$
We define the weighted versions $\Xi^a\ooo$ of $\ooo$ via $$\Xi^a \ooo(X_1, \dots X_m; Y_1, \dots, Y_n) = \Sigma^{(\sum_i a_{Y_i} - \sum_j a_{X_j})}\ooo(X_1, \dots X_m; Y_1, \dots, Y_n).$$ 

For every colour $X\in \Xscr$ we have an associated ``colour shift''
operator $\Xi_X$ with a formula similar to \eqref{eq:shiftedprop1}
where (de)suspensions are now only inserted in the in(out)puts
corresponding to the colour $X$. Note that the colour shift operators
only commute up to a sign.
\begin{example}
\label{ex:weighted_end}
For objects $V = (V_X)_{X\in \xxx}$  in a monoidal category $\Mscr$ enriched in $\sss$, we may consider the coloured endomorphism properad $\uEnd_{V}$ with
\[
\uEnd_{V}(X_1, \dots X_m; Y_1, \dots, Y_n)= \Mscr(V_{X_1} \otimes \dots \otimes V_{X_m}, V_{Y_1} \otimes \dots \otimes V_{Y_n}).
\]
If $\Mscr$ is closed under suspensions and desuspensions then
the definition of the weighted properad is such that 
$\Xi^a \uEnd_V \cong \uEnd_{V'}$ with $V'_X = \Sigma^{a_X}V_X$.

Below we will use a slight generalisation where $\B=\Mscr$ is a bicategory and where the $V_X$ are $1$-morphisms in $\B$ such that in addition the $\Hom$ in the definition of $\uEnd_V$ is restricted 
to composable $1$-morphisms (see \S\ref{sec:restricted_end} below). 
\end{example}

\subsection{$B_\infty$-algebras}
\begin{definition} Let $V$ be a graded vector space and denote by
  $T^c(V)$ its cotensor coalgebra, viewed as a coaugmented coalgebra  (with counit). A \emph{$\Xi B_\infty$-structure} on~$V$ is a
  dg-bialgebra structure $(T^c (V),\Delta,\epsilon,m,1,Q)$ on $T^c(V)$
  where the first half $(T^c(V),\Delta,\epsilon)$ is the standard coalgebra structure
  and $1\in T^c(V)$ is obtained from the coaugmentation.
\end{definition}
A $\Xi B_\infty$-structure is determined by operations
\begin{equation}
\label{eq:brace}
m^{s,t}:V^{\otimes s} \otimes V^{\otimes t}\r V
\end{equation}
of degree zero and operations
\begin{equation}
\label{eq:Q}
Q^r:V^{\otimes r} \r V
\end{equation}
of degree one. These operations should satisfy the additional conditions $Q^0=0$, $m^{0,1}=m^{1,0}=\Id_V$
and $m^{0,r}=m^{r,0}=0$ for $r\neq 1$. We may obtain $m$ and $Q$ from \eqref{eq:brace} and \eqref{eq:Q} using standard formulas:
\begin{align}
Q&=\sum_{p,q,r} \Id^{\otimes p}\otimes Q^r\otimes \Id^{\otimes q}\\\
m
&=\sum_{\sum_i s_i=s,\sum_j t_j=t} (m^{s_1,t_1}\otimes\cdots\otimes m^{s_u,t_u} )\circ \sigma^{s,t}_{(s_1,t_1)\cdots(s_u,t_u)}\label{eq:mdef}
\end{align}
where $\sigma^{s,t}_{(s_1,t_1)\cdots(s_u,t_u)}$ is the signed permutation which rearranges
\[
(v_1\otimes\cdots \otimes v_s)\otimes (w_1\otimes\cdots \otimes w_t)
\]
as
\begin{multline}
\label{eq:rearranging}
((v_1\otimes\cdots\otimes v_{s_1})\otimes (w_1\otimes\cdots \otimes w_{t_1}))
\otimes\cdots\\\cdots \otimes\
((v_{s_1+\cdots+s_{u-1}}\otimes\cdots\otimes v_{s})\otimes (w_{t_1+\cdots+t_{v-1}}\otimes\cdots \otimes w_{t})).
\end{multline}
We allow $s_i$ and $t_j$ to be zero in which case the
corresponding factor in \eqref{eq:rearranging} is~$1$.  Note that because of the conditions
on $m^{0,r}$, $m^{r,0}$ this is a rather degenerate situation.
\begin{remark} The standard definition of a bialgebra requires an underlying
\emph{symmetric} monoidal category. It follows that
 $B_\infty$-algebras can only be
defined in symmetric monoidal categories.
\end{remark}
\begin{remark} \label{rem:Ainfty} The biderivation $Q$ induces in particular a $\Xi A_\infty$ structure on $V$. This yields a morphism
of operads
\[
A_\infty\r B_\infty
\]
\end{remark}
We define the \emph{dot product} as.
\[
v\bullet w=m^{1,1}(v,w)
\]
\begin{example} \label{ex:lie}
One computes for $v,w\in V$
\begin{equation}
\label{eq:prelie}
m(v,w)=v\otimes w+(-1)^{|v||w|} w\otimes v+v\bullet w
\end{equation}
In particular the Lie bracket $[-,-]=m-m\circ \sigma_{12}$ restricts to a Lie bracket on $V$ defined
by
\[
[v,w]=v\bullet w-(-1)^{|v||w|} w\bullet v
\]
We obtain in particular a morphism of operads
\[
\Xi^{-1}\operatorname{Lie} \r B_\infty
\]
\end{example}
\begin{remark}
\label{rem:preLie}
The dot product itself is not associative. However it defines a so-called \emph{pre-Lie} structure on $V$.
I.e.\ one may show, using similar expressions as \eqref{eq:prelie} that the ``associator'' $(u\bullet v)\bullet w-u\bullet (v\bullet w)$ is anti-symmetric in $v,w$.
\end{remark}
$\Xi B_\infty$-algebras admit a natural version of twisting. Let $V$ be a
$\Xi B_\infty$-algebra and assume $\xi \in V$.

Then $\xi$ is a primitive
element and from this one deduces that $[\xi,-]$ (see Example \ref{ex:lie}) is both a derivation and a coderivation.
Consider the
\emph{Maurer-Cartan equation}
\begin{equation} \label{eqMC}
Q^1(\xi)+\xi\bullet \xi=0.
\end{equation}
\begin{proposition}\label{proptwist}
If $\xi \in V^1$ is a solution to the Maurer-Cartan equation \eqref{eqMC}, then changing $Q$ into
\begin{equation}
\label{eq:twist}
Q_\xi:=Q+[\xi,-]
\end{equation}
defines a new $\Xi B_\infty$-structure on $V$.
\end{proposition}

We will call the $\Xi B_{\infty}$-structure obtained in Proposition \ref{proptwist} the $\Xi B_\infty$-structure  \emph{twisted} by $\xi$.

\subsection{\boldmath The $\Xi B_\infty$-structure on properads}\label{parprop}
\label{sec:binftyproperads}

Let $\ooo$ be an asymmetric properad in~$\sss$ as in \S\ref{sec:generalities}.
The compositions in $\ooo$ give rise to two sided brace operations (called LR operations in \cite{merkulovvallette}) where the condition $|\underline{l}| = |\underline{m}|$ is dropped:
$$B = B^s_t = B^{\underline{m}, \underline{n}}_{\underline{k}, \underline{l}}: \bigotimes_{i = 1}^s \ooo(m_i, n_i) \otimes \bigotimes_{j = 1}^t \ooo(k_j, l_j) \lra \oplus_{k, n}\ooo(k, n)$$
with
\begin{multline}
\label{eq:moperations} 
B^s_t(\phi_1,\ldots,\phi_s,\psi_1,\ldots,\psi_t)
=
\sum (\Id^{\otimes a_0} \otimes \phi_1\otimes \Id^{\otimes a_1}
\otimes \cdots \otimes \Id^{\otimes a_{s-1}}\otimes \phi_s
\otimes  \Id^{\otimes a_{s}})\\
\circ
(\Id^{\otimes b_0} \otimes \psi_1\otimes \Id^{\otimes b_1}
\otimes \cdots \otimes \Id^{\otimes b_{t-1}}\otimes \psi_t\otimes \Id^{\otimes b_{t}}
)
\end{multline}
where the sum runs over all legal compositions (i.e.\ planar connected graphs).

Note that, in order for the total graph to be connected, every identity $\Id \in \ooo(1,1)$ 
that is inserted among the $\phi_i$ has to be connected to one of the $\psi_j$, and every identity inserted among the $\psi_j$ has to be connected to one of the $\phi_i$. Hence, the total number of identities is bounded by $|\underline{l}| + |\underline{m}|$ and the sum in \eqref{eq:moperations} is finite and more precisely lands in 
$$\oplus_{k = |\underline{k}|}^{|\underline{k}| + |\underline{m}|} \oplus_{n = |\underline{n}|}^{|\underline{n}| + |\underline{l}|} \ooo(k, n).$$

Consider the $\sss$-objects
\begin{align*}
\CC_{\oplus}(\ooo) &= \bigoplus _{k, n}\ooo(k, n)\\
\CC_{\Pi}(\ooo) &= \prod_{k, n}\ooo(k, n)
\end{align*}
\begin{proposition}\cite[Prop. 9]{merkulovvallette} \label{proppropBinfty}
For a properad $\ooo$, $\CC_{\oplus}(\ooo)$ has the structure of a $\Xi B_{\infty}$-algebra. If $\Oscr$ has bounded connectivity then this also holds for $\CC_{\Pi}(\ooo)$.
\end{proposition}

\begin{proof}
We sketch the proof. First consider $\CC_{\oplus}(\ooo)$.
 In case $\sss = G(k)$ we put $Q = 0$ and in case $\sss = C(k)$ the only non-zero component of $Q$ is $Q^1$ which is the induced differential on $\CC_{\oplus}(\ooo)$.

We now  let $m^{s,t}$ be obtained by extending $B^s_t$ linearly to $\CC_{\oplus}(\ooo)$
and we define
$m$ by \eqref{eq:mdef}. We have to verify that $m$ commutes with $Q$ and is associative. Commutation
with $Q$ is obvious so we only have to 
check associativity. We first observe that
\[
m(\phi_1,\ldots,\phi_s,\psi_1,\ldots,\psi_t)
\]
has a similar formula as \eqref{eq:moperations} but now we also allow disconnected graphs.
The formula
\[
m(m\otimes \Id)=m(\Id \otimes m)
\]
can then be checked by observing that both sides contain the same graphs. 

\medskip

To extend this argument to $\CC_{\Pi}(\ooo)$ we have to extend the brace operations
\[
B^s_t: \CC_{\Pi}(\ooo)^{\otimes s} \otimes \CC_{\Pi}(\ooo)^{\otimes t} \lra \CC_{\Pi}(\ooo).
\]
To do this it clearly suffices that the number of brace operations $B^{\underline{m}', \underline{n}'}_{\underline{k}', \underline{l}'}$ with non-zero contribution of the image to a fixed $\ooo(k,n)$ is finite.
This follows from bounded connectivity as the number of possible graphs yielding a non-zero contribution is finite.
\end{proof}
Below our main emphasis will be $\CC_{\Pi}(\ooo)$. So we write
\[
\CC(\ooo)=\CC_{\Pi}(\ooo)
\]
For a coloured operad of bounded connectivity we write
\[
\CC(\ooo)=\CC_{\Pi}(\ooo) = \prod_{x,y}\ooo(x,y)
\]
where the product is over all colour sequences of inputs and outputs.
Proposition \ref{proppropBinfty} extends trivially
to coloured properads and we will use it as such. 

\section{$A_\infty$-algebras, coalgebras and modules}\label{sec:Ainfty}

\subsection{$A_\infty$-structures}
Let $(\aaa, \otimes, I)$ be a $k$-linear monoidal category.
 It will be convenient to embed $C(\Ascr)$ into $C^\ast(\Ascr):=\Mod(\Ascr)$ where $\Mod(\Ascr)$
consists of the contravariant dg-functors $\Ascr\r C(k)$. We will not equip $C^\ast(\Ascr)$ with
a monoidal structure, but we will use the fact that it is a module over the monoidal category
$C(k)$. Unfortunately $C(\Ascr)$ is not a monoidal category if $\Ascr$ is not closed under coproducts. Therefore we will work with the
category of right bounded complexes $C^-(\Ascr)$. If $\Ascr$ admits coproducts then we may work with $C(\Ascr)$ as well.

\medskip

In addition, $C^-(\aaa)$ is naturally enriched over $\sss = C(k)$.  In
this section we extend some well-known notions in $C(k)$ to
$C^-(\Ascr)$.  The easy standard proofs go through unmodified.
The differential on the $\Hom$-spaces in $C^-(\Ascr)$ will be denoted by $Q^1$.
We have $Q^1=[d,-]$ where $d$ is the differential of the complexes in $C^-(\Ascr)$.

\medskip

For $A \in C^-(\Ascr)$, consider the endomorphism operad $\End_A$ in $\sss$ with $$\End(A)(n) = \Hom_{\aaa}(A^{\otimes n}, A).$$ The Hochschild object $\CC_{\Hoch}(A)$ of $A$ is defined by
\begin{equation}\label{Hochobj}
\Sigma \CC_{\Hoch}(A) = \CC(\Xi \End_A).
\end{equation}
For $C \in G(A)$, consider the endomorphism co-operad $\End^{\mathrm{co}}_C$ in $\sss$ with $\End^{\mathrm{co}}_C(n) = \Hom_{\aaa}(C, C^{\otimes n})$. The co-Hochschild object $\CC_{\coHoch}(C)$ of $C$ is defined by
\begin{equation}\label{coHochobj}
\Sigma \CC_{\coHoch}(C) = \CC(\Xi^{-1} \End^{\mathrm{co}}_C).
\end{equation}
By Proposition \ref{proppropBinfty}, both $\CC_{\Hoch}(A)$ and $\CC_{\coHoch}(C)$ are $B_{\infty}$-algebras. 

We recall the following
\begin{definition}
\begin{enumerate}
\item A \emph{$ A_{\infty}$-algebra structure} on $A$ is a degree $1$
  element $\xi \in \Sigma \CC_{\Hoch}(A)$ with
\begin{equation}
\label{eq:mc}
 Q^1(\xi) + \xi \bullet \xi = 0
\end{equation} 
and
  $0 = \xi_0 \in \Hom_{\aaa}(I, \Sigma A)$,
where the $\xi_i\in \Hom_{\aaa}((\Sigma A)^{\otimes i},\Sigma A)$ are the components of $\xi_i$.
It is easy to see that \eqref{eq:mc} implies that $\xi_2$  is closed.
\item An $ A_{\infty}$-algebra $A$ has a \emph{homotopy unit}
  $\eta: I \lra A$ provided that $\eta$ is closed of degree zero and there are, unspecified, homotopies
  $\xi_2(\eta \otimes \id) \sim \id$ and $\xi_2(\id\otimes \eta) \sim \id$.
\item An \emph{$ A_{\infty}$-coalgebra structure} on $C$ is a degree
  $1$ element $\xi \in \Sigma \CC_{\coHoch}(C)$ with
  $Q^1(\xi) + \xi \bullet \xi = 0$ such that
  $0 = \xi_0 \in \Hom_{\aaa}(\Sigma^{-1} C, I)$,
\item An $A_{\infty}$-coalgebra structure on $C$ has a \emph{homotopy
    counit} $\epsilon: C \lra I$ provided that $\epsilon$ is closed of degree zero and there are, unspecified, homotopies
  $(\epsilon \otimes \id)\xi_2 \sim \id$ and
  $(\id \otimes \epsilon)\xi_2 \sim \id$.
\end{enumerate}
\end{definition}
In the rest of this section we will follow tradition and write an $A_\infty$-(co)algebra structure
as $b_n=\xi_n$ for $n\ge 2$. In addition we let $b_1=d+\xi_1$.
Finally we put $b_0=0$. For an $A_\infty$-algebra structure on $A$ we should then have
on $(\Sigma A)^{\otimes n}$
\begin{equation}
\label{eq:Ainfty}
\sum_{a+c+b=n} b_{a+1+b}\circ (\id^{\otimes a} \otimes b_c\otimes \id^{\otimes b})=0
\end{equation} 
and for an $A_\infty$-coalgebra structure on $C$ we should have on $\Sigma^{-1} C$ 
\begin{equation}
\label{eq:coAinfty}
\sum_{a+c+b=n} (\id^{\otimes a} \otimes b_b\otimes \id^{\otimes b})\circ b_{a+1+b}=0.
\end{equation}
\begin{lemma}
Let $C$ be an $ A_\infty$-coalgebra in $C^-(\Ascr)$. Then $C^\vee:=\Sigma^{-1}\Hom(\Sigma^{-1}C,I)\allowbreak\cong \Hom(C,I)$
is an $A_\infty$-algebra in $C(k)$ using the formula 
\begin{equation}
\label{eq:dual}
b^{C^\vee}_n(\phi_1\otimes \cdots\otimes \phi_n)=(-1)^{\sum_i |\phi_i|} \mu\circ (\phi_1\otimes \cdots\otimes \phi_n) \circ b^C_n
\end{equation}
for $\phi_i\in \Sigma(C^\vee)=\Hom_{\Ascr}(\Sigma^{-1}C,I)$, where
$\mu$ is a generic notation for morphisms of the type $I^{\otimes a}\otimes A\otimes I^{\otimes b}\xrightarrow{\cong} A$.
 Moreover if $C$ has a homotopy counit $\epsilon$ then $\epsilon$ becomes a homotopy unit in $C^\vee$.
\end{lemma}
\begin{proof}
It is easy to verify that the equation \eqref{eq:coAinfty} for $\Sigma^{-1}C$  becomes \eqref{eq:Ainfty} for $\Sigma(C^\vee)$. 
The claim about (co)units is also easy.
\end{proof}

Recall that if $C$ is an $A_\infty$-coalgebra in $C^-(\Ascr)$ then an $A_\infty$-comodule $M$ over $C$ in $C^-(\Ascr)$ is 
an object in $C^-(\Ascr)$ together with morphisms of degree one
\[
b_n:M\r \underbrace{\Sigma^{-1}C\otimes\cdots\otimes \Sigma^{-1}C}_{n-1}\otimes M
\] 
satisfying an obvious analogue of \eqref{eq:coAinfty}.  We say that $M$ is homotopically counital if $(\epsilon\otimes \Id_N)\circ b_2$
is homotopic to the identity $M\r M$.

Similarly if $A$ is an $A_\infty$-algebra in $C^-(\Ascr)$ then an $A_\infty$-module $M$ over $A$ in $C^-(\Ascr)$ is 
an object in $C^-(\Ascr)$ together with morphisms of degree one
\begin{equation}
\label{eq:universal}
b_n:\underbrace{\Sigma A\otimes\cdots\otimes \Sigma A}_{n-1}\otimes M\r 
M
\end{equation} 
satisfying the obvious analogue of \eqref{eq:Ainfty}. We say that $M$
is homotopically unital if $b_2\circ (\eta\otimes \Id_N)$ is homotopic
to the identity $M\r M$. What we actually need below is the variant
where $A$ is an $A_\infty$-algebra in $C(k)$ and $M\in C^-(\Ascr)$. In
that case we view \eqref{eq:universal} as a morphism in
$C^\ast(\Ascr)$ where we use the $C(k)$ action on $C^\ast(\Ascr)$.
We may specify $b_n$ by specifying for each $sa_i\in \Sigma A$
the corresponding degree one map $b_n(sa_1\otimes \ldots\otimes sa_{n-1}\otimes -):M\r M$
\begin{lemma} \label{eq:verify}
Assume that $C$ is an $A_\infty$-coalgebra in $C^-(\Ascr)$. Let $M$ be an $A_\infty$-module over $C$. Then $M$ becomes an $A_\infty$-module
over $C^\vee$ via
\[
b_n(\phi_1\otimes\cdots\otimes \phi_n\otimes-)=(-1)^{\sum_i |\phi_i|}\mu\circ (\phi_1\otimes\cdots \otimes
 \phi_{n-1}\otimes \Id_M)\circ b_n
\]
where $\phi_i\in \Sigma(C^\vee)=\Hom(\Sigma^{-1} C,I)$. 
If $M$ is homotopically counital over $C$ then it is homotopically unital over $C^\vee$.
\end{lemma}
If $A$, $B$ are $A_\infty$-algebras in $C(k)$ then an $A_\infty$-morphism $f:A\r B$ is a sequence of maps of degree 0 for $n\ge 1$
\[
f_n:(\Sigma A)^{\otimes n} \r \Sigma B
\]
such that
\begin{equation}
\label{eq:Ainftymorphism}
\sum_{\sum_{i=1}^k u_i=n} b_{k} \circ (f_{u_1}\otimes\cdots \otimes f_{u_k})=
\sum_{a+c+b=n}f_{a+1+b}\circ (\Id^{\otimes a}\otimes b_c\otimes \Id^{\otimes b})
\end{equation}
For use below we say that $f$ is strict if $f_i=0$ for $i\ge 2$.
An $A_\infty$-morphism is homotopically unital if $f_1\circ \eta_A$ is homotopic to $\eta_B$.
The following lemma is standard.
\begin{lemma}
If $A$ is an $A_\infty$-algebra in $C(k)$ and $M$ is an $A_\infty$-$A$-module in $C^-(\Ascr)$ (see above) then
\begin{multline*}
f_n:\Sigma A\otimes\cdots\otimes \Sigma A \r \Sigma \End_k(M):\\
sa_1\otimes\cdots \otimes sa_n
\mapsto s(b_{n+1}(sa_1\otimes \ldots\otimes sa_n\otimes -))
\end{multline*}
is an $A_\infty$-morphism where the dg-algebra $\End(M)$ is regarded as an $A_\infty$-algebra as in Example \ref{ex:assoc}. If $A$ and $M$ are homotopically unital
then $f$ is homotopically unital. Moreover this construction is reversible.
\end{lemma}
\begin{corollary}
\label{cor:action}
Assume that $C$ is an $A_\infty$-coalgebra in $C^-(\Ascr)$ then there is a canonical 
$A_\infty$-morphism
\[
f:C^\vee\r \Hom(C,C)
\]
such that if $C$ is homotopically counital then $f$ is homotopically unital, and such that moreover the composition
\[
C^\vee\r \Hom(C,C)\xrightarrow{\epsilon\circ-} \Hom(C,I)\cong C^\vee 
\]
is homotopic to the identity.
\end{corollary}
\begin{proof} 
The first claim follows from Lemma \ref{eq:verify} since $C$ is an $A_\infty$-comodule over $C$. The second claim is a direct verification.
\end{proof}

\subsection{The $A_\infty$-structure on the Hochschild complex of an $A_\infty$-coalgebra}
By Propositions 
\ref{proptwist} 
 and 
\ref{proppropBinfty}
we have:
\begin{proposition}\label{Binftycohoch}
\begin{enumerate}
\item Let $A \in C^-(\aaa)$ be endowed with an $A_{\infty}$-algebra structure $\xi$. Then the graded $k$-module $\Sigma \CC_{\Hoch}(A)$ becomes a $\Xi B_{\infty}$-algebra with differential $Q^1 + [\xi, -]$ for the commutator with respect to the dot product.
\item  Let $C \in C^-(\aaa)$ be endowed with an $A_{\infty}$-coalgebra structure $\xi$. Then the graded $k$-module $\Sigma \CC_{\coHoch}(C)$ becomes a $\Xi B_{\infty}$-algebra with differential $Q^1 + [\xi, -]$ for the commutator with respect to the dot product.
\end{enumerate}
\end{proposition}
Note that a $B_\infty$-algebra is in particular an $A_\infty$-algebra by Remark \ref{rem:Ainfty}. If $C$ is an $A_\infty$-coalgebra then the $A_\infty$-structure on  $\CC_{\coHoch}(C)$ is
given by
\begin{align*}
b_1=[b^C,-]
\end{align*}
and for $n\ge 2$
\begin{align*}
b_n(\phi_1\otimes\cdots\otimes \phi_n)=(-1)^{\sum_i |\phi_i|}(\phi_1\otimes\cdots\otimes\phi_n)\circ b^C_n
\end{align*}
where $\phi_i\in \Sigma\CC_{\coHoch}(C)$. From the formula \eqref{eq:coAinfty} It follows that the projection
\begin{equation}
\label{eq:p}
p:\CC_{\coHoch}(C)\r \CC^0_{\coHoch}(C)= C^\vee.
\end{equation}
is a strict $A_\infty$-morphism.
If $C$ is homotopically counital then it is easy to see that this is the case for $\CC_{\coHoch}(C)$, and moreover $p$ is also counital.

\section{The $B_{\infty}$-structure on $\REnd$}\label{sec:BinftyREnd}
\subsection{Introduction}
In this section we assume that $\Ascr=(\Ascr,\otimes,I)$ is a $k$-linear abelian category with enough projectives. We do not assume that $\otimes$ is exact. If $M$ is an object in $\Ascr$ and $P^\bullet$ is a projective resolution
if $M$ then he usual triangular construction shows that $\End_{\Ascr}(P^\bullet)$  is independent of $P^\bullet$ in $\Ho(\dgAlg)$, the homotopy category
of dg-algebras. Therefore we write $\REnd_\Ascr(M)=\End_{\Ascr}(P^\bullet)$. The following is the main result of this paper.
\begin{theorem} \label{th:mainth}
Assume that $\Ascr$ has enough projectives and that $P\otimes-$ is exact for $P$ projective. Let $\bold{A}=\REnd_{\Ascr}(I)$. Then there exists a $B_\infty$-algebra $\bold{B}$ together
with a $A_\infty$-quasi-isomorphism $\bold{B}\r \bold{A}$.
\end{theorem}
\begin{remark} It is not hard to deduce from the proof below that  $\bold{B}$ is canonical in the homotopy category of $B_\infty$-algebras.
\end{remark}
\subsection{Proofs}
The following is proven in \cite[Thm. 3.1.1]{negronvolkovwitherspoon} for~$\aaa$ the category of bimodules over an algebra. The proof in the general case goes along the same lines.
\begin{proposition}\label{propcoalgP}
  Let $\epsilon: P^\bullet \lra I$ be a projective resolution such
  that for every $n$,
  $\epsilon^{\otimes n}: P^{\bullet\otimes n} \lra I^{\otimes n} \cong I$ is
  a resolution (not necessarily projective). There is an $A_{\infty}$-coalgebra structure on~$P^\bullet$
  with homotopy counit $\epsilon$, lifting the trivial
  coalgebra-structure on $I$.
\end{proposition}
\begin{proof}
Since $I$ is a coalgebra in $\Ascr$ and $\Ext^i_{\aaa}(I,I) = 0$ for $i < 0$ we can lift the coalgebra structure of $I$ to an $A_{\infty}$-coalgebra structure on $P^\bullet$.
The components of this structure are obtained by solving \eqref{eq:mc} inductively (thinking $\xi=\sum_{n\ge 2} \xi_n$) using the pre-Lie structure\footnote{In characteristic different
from $2$ one may avoid
using the pre-Lie structure and use the simpler Lie algebra structure instead by writing \eqref{eq:mc} as $Q^1(\xi)+(1/2)[\xi,\xi]=0$}.
(cfr.\ Remark \ref{rem:preLie}) on $\Sigma \CC_{\coHoch}(P^\bullet)$. More precisely we let $\xi_2:\Sigma^{-1}P^\bullet\r \Sigma^{-1}P^\bullet\otimes \Sigma^{-1}P^\bullet$ be obtained by lifting the
identity $\Sigma^{-1}I\r \Sigma^{-1}I\otimes \Sigma^{-1}I$ (this is possible since $P^\bullet$ consists of projectives). Next we find that for each $n>2$ we have to solve an equation
of the form $Q^1(\xi_n)=\pi_n$ where $\pi_n$ is  a morphism $\Sigma^{-1}P^\bullet\r (\Sigma^{-1}P^\bullet)^{\otimes n}$ satisfying $Q^1(\pi_n)=0$. Using the conditions on~$P^\bullet$ we find 
that~$\pi_n$ represents an element of $\Ext^{-n+1}_{\Ascr}(I,I)=0$. This allows us to define $\xi_n$  (see also \cite[Thm. 3.1.1]{negronvolkovwitherspoon}).
\end{proof}
\begin{remark}
\label{rem:exact}
If the projectives $P_i$ are such that $P_i \otimes -$ is exact, then the condition upon $P$ in Proposition \ref{propcoalgP} is fulfilled.
\end{remark}
We now state some results on $A_\infty$-coalgebras in $C^-(\Ascr)$.
\begin{corollary} 
\label{cor:quasi}
Assume that $C$ is a homotopically counital $A_\infty$-coalgebra in $C^-(\Ascr)$ which is a right bounded complex of projectives.
Assume furthermore that the counit is a quasi-isomorphism. Then there is a homotopically unital $A_\infty$-quasi-isomorphism 
\[
C^\vee \r \End_{\Ascr}(C)\cong \REnd_{\Ascr}(I)
\]
\end{corollary}
\begin{proof} This follows from Corollary \ref{cor:action}.
\end{proof}
\begin{proposition}
\label{prop:mainprop}
 Assume that  $C$ is a homotopically counital $A_\infty$-coalgebra in $C^-(\Ascr)$ which is a right bounded complex of projectives such that the counit is a quasi-isomorphism. 
Assume furthermore that for any $n$, $C^{\otimes n}$ is a resolution (not necessarily projective) of \hbox{$I^{\otimes n}=I$}. Then the projection $p:\CC_{\coHoch}(C)\r C^\vee$ (see \eqref{eq:p}) 
is a homotopically unital $A_\infty$ quasi-isomorphism.
\end{proposition}

\begin{proof}
Put $\CC=\CC_{\coHoch}(C)$. We filter $\Sigma\CC$
by 
\[
F^p \Sigma \CC=\prod_{m\ge p} \Hom(\Sigma^{-1}C,(\Sigma^{-1}C)^{\otimes m})
\]
The differential on
$\Sigma\CC$ is given by $[d+\xi,-]$
 and it is compatible with
this filtration. So it it is sufficient that 
\[
p:\Sigma \CC\r \Sigma (C^\vee)
\]
induces in isomorphism on $E_2$-pages. The $E_1$-page amounts
to taking the differential for $[d+\xi_1,-]$. The $E_1$-page
for $\Sigma(C^\vee)$ is 
\[
E^{pq}_1=
\begin{cases}
H^q(\Hom_{\Ascr}(\Sigma^{-1}C,I)) & \text{if $p=0$}\\
0&\text{otherwise}
\end{cases}
\]
with zero differential.
Using the hypotheses on $C$ we find that the $E_1$-page for $\Sigma \CC$ is 
\[
E^{pq}_1=H^q(\Hom_{\Ascr}(\Sigma^{-1}C,I)) 
\]
for all $p,q$ and the differential is given by $d_2=[\xi_2,-]$.
The morphism between the $E_1$-pages is still the projection
on the first column.

Using the homotopy
  counit axiom, it is easily seen that $d^{pq}_2 = 0$ for $p$
  even and~$d^{pq}_2$ is an isomorphism for $p$ odd (see the
  proof of Proposition \ref{propalt}, replacing the counit axiom by
  the homotopy counit axiom). Hence $\Sigma C^\vee$ and $\Sigma \CC$
have identical $E_2$-pages.
This finishes the proof.
\end{proof}
\begin{proof}[Proof Theorem \ref{th:mainth}]
We take a projective resolution $P^\bullet \r I$ consisting of projectives $P$ such that $P\otimes-$ is exact. Then we use Proposition \ref{propcoalgP} and Remark \ref{rem:exact}
to equip $P^\bullet$ with an $A_\infty$-coalgebra structure. We put $\bold{B}=\CC_{\coHoch}(P^\bullet)$. By combining Proposition \ref{prop:mainprop} and Corollary \ref{cor:quasi} with $C=P^\bullet$
we obtain an $A_\infty$-quasi-isomorphism $\bold{B} \r \End_{\Ascr}(P^\bullet)=\bold{A}$.
\end{proof}

\section{Semi-coalgebras}
\label{sec:semico}
\subsection{Introduction}
The rest of this paper is devoted to proving an agreement property
between the usual $B_\infty$-structure on the Hochschild complex of a
$k$-algebra (see Proposition \ref{Binftycohoch}(1)), and the one
furnished by Theorem \ref{th:mainth} which uses coalgebras in an essential way.
To this end  we introduce below a kind of structure which is a mixture of a
category and a cocategory. More precisely we partition the vertices of
a graph with edges in $\Sscr$ in two sets consisting respectively of
``objects'' and ``co-objects''. The objects span a category, the
co-objects span a cocategory and the $\Hom$-spaces connecting objects
to co-objects, and vice versa, are modules on one side and comodules
on the other side. Such a structure could reasonably be called a
semi-co-category.  We show that it is possible to associate a
Hochschild complex to a semi-co-category which has the structure of a
$\Xi B_\infty$-algebra.

For simplicity we explain everything in the case that there is a single object $\alpha$ and a single co-object $\gamma$. The general case is similar.
\subsection{Setup}
Let $\B$ be a bicategory enriched in $\sss$, in the sense that for objects $\beta, \beta' \in \B$ we have a category $\B(\beta, \beta')$ enriched over $\sss$ and the compositions $\circ: \B(\beta', \beta'') \times \B(\beta, \beta') \lra \B(\beta, \beta'')$ are $\sss$-bifunctors equipped with the standard associativity data (which will be suppressed below). Note in particular that $\B(\beta)$ is a monoidal category.  Since
we want to emphasise the monoidal nature of our constructions we write $X\otimes Y=Y\circ X$.

The objects of $\B$ will become the objects or co-objects in a semi-co-category. Hence in accordance with the simplifying hypothesis made in the introduction, we assume $\Ob(\B) = \{\alpha, \gamma\}$. 
\subsection{The restricted endomorphism properad}
\label{sec:restricted_end}
Let the bicategory $\B$ be as before. We consider four colours $\underline{A} = (\alpha, \alpha)$, $\underline{M} = (\alpha, \gamma)$, $\underline{C} = (\gamma, \gamma)$ and $\underline{N} = (\gamma, \alpha)$. To a choice of objects $A \in \B(\alpha, \alpha)$, $M \in \B(\alpha, \gamma)$, $C \in \B(\gamma, \gamma)$ and $N \in \B(\gamma, \alpha)$ we associate the \emph{coloured endomorphism properad} $\eee nd = \eee nd_{A, M, C, N}$ in $\sss$ with
\begin{equation}\label{eqend}
\eee nd(\underline{X}_1, \dots, \underline{X}_n; \underline{Y}_1, \dots, \underline{Y}_m) = \Hom(X_1 \otimes \dots \otimes X_n, Y_1 \otimes \dots \otimes Y_m)
\end{equation}
for $X_i, Y_i \in \{A, M, C, N\}$ whenever the expression on the right makes sense and zero otherwise. Here, we require in particular that for $\underline{X}_i = (x^i_0, x^i_1)$ and $\underline{Y}_i = (y^i_0, y^i_1)$ we have $x^1_0 = y^1_0 = \beta_0$ and $x^n_1 = y^m_1 = \beta_1$ and the $\Hom$ in \eqref{eqend} is computed in $\B(\beta_0, \beta_1)$.

Next, we introduce the \emph{restricted endomorphism properad}
$\eee nd^{\ast} = \eee nd^{\ast}_{A, M, C, N}$ where we impose further
conditions upon the sequences $({\underline{X}}_i)_i$ and $({\underline{Y}}_i)_i$ in order for
$\eee nd^{\ast}$ to take the possibly non-zero value
\eqref{eqend}. More precisely, we require that all the values $x^i_j$
are equal to $\alpha$ except possibly $x^1_0$ and $x^n_1$, and all the
values $y^i_j$ are equal to $\gamma$ except possibly $y^1_0$ and
$y^n_1$. Concretely, the allowable input colour sequences are
\begin{itemize}
\item subsequences of $(\underline{N}, \underline{A}, \dots, \underline{A}, \underline{M})$
\item $(\underline{C})$
\end{itemize}
Here, we make the convention that the empty input sequence $()$ has colour $\underline{A}=(\alpha,\alpha)$, and in \eqref{eqend} we interpret it as $A^{\otimes 0}$ which, by convention, is the tensor unit of $\B(\alpha, \alpha)$.
Similarly, the allowable output colour sequences are
\begin{itemize}
\item subsequences of $(\underline{M}, \underline{C}, \dots, \underline{C}, \underline{N})$
\item $(\underline{A})$
\end{itemize}
This time, we make the convention that the empty output sequence $()$ has colour $\underline{C}=(\gamma,\gamma)$, and in \eqref{eqend} we interpret it as $C^{\otimes 0}$ which is the tensor unit of $\B(\gamma, \gamma)$.
Taking begin and end points into account, this leaves us with the following allowable combinations:
\begin{itemize}
\item[(a)] $(\underline{A},  \dots, \underline{A}, \underline{M}\,; \underline{M}, \underline{C}, \dots, \underline{C})_{\alpha\gamma}$ 
\item[(b)] $(\underline{N}, \underline{A}, \dots, \underline{A}\, ; \underline{C},  \dots, \underline{C}, \underline{N})_{\gamma\alpha}$ 
\item[(c)] $(\underline{A}, \dots, \underline{A}\,; \underline{M}, \underline{C}, \dots, \underline{C}, \underline{N})_{\alpha\alpha}$ 
\item[(d)] $(\underline{N}, \underline{A}, \dots, \underline{A}, \underline{M}\,; \underline{C}, \dots,  \underline{C})_{\gamma\gamma}$ 
\item[(e)] $(\underline{A},  \dots,  \underline{A}\,; \underline{A})_{\alpha\alpha}$
\item[(f)] $(\underline{C}\,; \underline{C},  \dots, \underline{C})_{\gamma\gamma}$
\end{itemize} 
where the variable length subsequences denoted by $\underline{A},\ldots,\underline{A}$ and $\underline{C},\ldots,\underline{C}$ may be empty.

To prove that $\uEnd^{\ast}$ is a suboperad of $\uEnd$ we show that is it is closed under elementary compositions (see \S\ref{sec:generalities}). By the restrictions induced by the 
colours it turns out that all elementary compositions have connection arity one. 

We list the different elementary
compositions $\phi\circ \psi$ in the following table, with the types for $\phi$
listed in the left column and the types for $\psi$ listed in the
upper row. We also indicate the colour of the connecting edge.
$$\begin{array}{c|cccccc} \label{arrayprop}
{\circ} & a & b & c & d & e & f \\
\hline
a & a_{\underline{M}} & 0 & c_{\underline{M}} & 0 & a_{\underline{A}} & 0 \\ 
b & 0 & b_{\underline{N}} & c_{\underline{N}} & 0 & b_{\underline{A}} & 0 \\ 
c & 0 & 0 & 0 & 0 & c_{\underline{A}} & 0 
\\ d & d_{\underline{M}} & d_{\underline{N}} & a_{\underline{N}}/b_{\underline{M}} & 0 & d_{\underline{A}} & 0 
\\ e & 0 & 0 & 0 & 0 & e_{\underline{A}} & 0 
\\ f & a_{\underline{C}} & b_{\underline{C}} & c_{\underline{C}} & d_{\underline{C}} & 0 & f_{\underline{C}} 
\end{array}$$

\begin{proposition}\label{proplist}
The following hold true:
\begin{enumerate}
\item $\eee nd^{\ast}$ is a sub-properad of $\eee nd$;
\item $\uEnd^\ast$ has bounded connectivity (in fact all connection arities are 1).
\item $\eee nd^{\ast}$ has sub-properads $\aaa$ generated by elements of type (e), $\ccc$ generated by elements of type (f) and $\mmm$ generated by elements of the four other types, and these satisfy $\eee nd^{\ast} = \aaa \oplus \mmm \oplus \ccc$;
\item $\mmm$ is a two-sided properadic ideal in $\eee nd^{\ast}$;
\item in order for a non-trivial composition to land in $\aaa$, all components have to be in $\aaa$;
\item in order for a non-trivial composition to land in $\ccc$, all components have to be in $\ccc$.
\end{enumerate}
\end{proposition}

\begin{proof}
These observations immediately follow from inspection of the composition table.
\end{proof}
In order to define semi-coalgebra structures, we consider the weight function $w$ which assigns the weights $(1, 0, -1, 0)$ to the colours $(\underline{A}, \underline{M}, \underline{C}, \underline{N})$.
We define the \emph{semi-co-Hochschild object} of $(A,M,C,N)$ by
\begin{equation}
\label{def:schoch}
\Sigma \CC_{\scHoch}(A,M,C,N) = \CC(\Xi^w \eee nd^{\ast}_{A,M,C,N}).
\end{equation}
If $N=0$ then we omit it from the notations. E.g.\ we
write $\CC_{\scHoch}(A,M,C)=\CC_{\scHoch}(A,M,C,N)$.
\begin{definition}
\label{def:semico}
A \emph{semi-coalgebra structure} on $(A,M,C,N)$ is a degree $1$ element $\xi \in \Sigma \CC_{\scHoch}(A,M,C,N)$ satisfying the Maurer Cartan equation (see \eqref{eqMC})
\begin{equation}
\label{eq:mc2}
Q^1\xi+\xi \bullet \xi = 0
\end{equation}
such that all components of $\xi$ in $\Hom(X_1 \otimes \dots \otimes X_m, Y_1 \otimes \dots \otimes Y_n)$ are zero except for $(m,n)=(1,2)$ and $(m,n)=(2,1)$.
\end{definition}
\begin{remark} Note that as explained in \S\ref{sec:binftyproperads}
$Q^1$ in \eqref{eq:mc2} is obtained by extending the differential on the 2-morphisms in $\B$ (which is zero if $\sss=G(k)$). So it makes sense to
denote $Q^1$ simply by ``$d$''.
\end{remark}
\begin{remark}
By relaxing the conditions on arities, one naturally obtains $\infty$-versions and even curved versions of the concept of a semi-coalgebra.
\end{remark}
From Propositions \ref{proppropBinfty} and \ref{proptwist} we immediately obtain:
\begin{proposition}
\label{prop:twist2}
Suppose $(A,M,C,N)$ as above is endowed with a semi-coalgebra structure $\xi$. Then $\Sigma \CC_{\scHoch}(A,M,C,N)$ becomes a $\Xi B_{\infty}$-algebra with differential $Q^1+[\xi, -]$ for the commutator with respect to the dot product.
\end{proposition}
\begin{remark}
\label{rem:twist}
Sometimes we write $\Sigma \CC_{\scHoch,\xi}(A,M,C,N)$ for $\Sigma\CC_{\scHoch}(A,M,C,N)$ when equipped with the ``twisted'' $\Xi B_\infty$-structure obtained from $\xi$, exhibited in Proposition \ref{prop:twist2}. We use similar self explanatory notations such as $\Sigma \CC_{\Hoch,\xi}(A)$ and $\Sigma\CC_{\coHoch,\xi}(C)$.
\end{remark}
\medskip

When looking at the subproperads $\aaa$ and $\ccc$ of $\eee nd^{\ast}$, we observe that the corresponding subproperads of $\Xi^w \eee nd^{\ast}$ are given by the weighted endomorphism operad $\Xi \End_A$ associated to the object $A \in \B(\alpha, \alpha)$ and the weighted co-operad $\Xi^{-1}\End_C$ associated to the object $C \in \B(\gamma, \gamma)$. We thus recognise the corresponding Hochschild object $\Sigma \CC_{\Hoch}(A) = \CC(\Xi \End_A)$ from \eqref{Hochobj} and co-Hochschild object $\Sigma \CC_{\coHoch}(C) = \CC(\Xi^{-1} \End^{\mathrm{co}}_C)$ from \eqref{coHochobj}.

\begin{proposition}
\label{prop:piApiC}
The two canonical projections
$$\pi_A: \Sigma \CC_{\scHoch}(A,M,C,N) \lra \Sigma \CC_{\Hoch}(A)$$
$$\pi_C: \Sigma \CC_{\scHoch}(A,M,C,N) \lra \Sigma \CC_{\coHoch}(C)$$
are morphisms of $\Xi B_{\infty}$-algebras. If $\xi$ is a
semi-coalgebra structure on $(A,M,C,N)$, then $\pi_A(\xi)$ defines a
dg-algebra structure on $A$ and $\pi_C(\xi)$ defines a dg-coalgebra
structure on $C$. For the three resulting twisted
$\Xi B_{\infty}$-structures, $\pi_A$ and $\pi_C$ become morphisms of
$\Xi B_{\infty}$-algebras.
\end{proposition}
\begin{proof}
This easily follows from Proposition \ref{proplist}.
\end{proof}
A semi-coalgebra structure consists of 8 operations of degree one 
which may be written in terms of operations in $\uEnd^\ast$ (see \S\ref{sec:shifted}).
When applied to operations we define $\Xi^w$ as $\Xi_{\underline{A}} \Xi^{-1}_{\underline{C}}$.
Note however  $\Xi_{\underline{A}}$ and $\Xi_{\underline{C}}$ commute when acting on $\uEnd^\ast$
since $\underline{A}$ and $\underline{C}$ are never together colours of input or outputs
of a non-zero operation.

The resulting operations in $\uEnd$ are as in the following table. 
\begin{equation}
\label{eq:operations}
\begin{aligned}
\xi_{\underline{A},\underline{A};\underline{A}}&:=\Xi^w(m)&&\text{ with }m\in Z^0\Hom(A\otimes A,A),\\
\xi_{\underline{C};\underline{C},\underline{C}}&:=\Xi^w(\Delta)&&\text{ with }\Delta\in Z^0\Hom(C,C\otimes C),\\
\xi_{\underline{A},\underline{M};\underline{M}}&:=\Xi^w(\mu_M)&&\text{ with }\mu_M\in Z^0\Hom(A\otimes M,M),\\
\xi_{\underline{M};\underline{M},\underline{C}}&:=-\Xi^w(\delta_M)&&\text{ with }\delta_M\in Z^0\Hom(M,M\otimes C),\\
\xi_{\underline{N},\underline{A};\underline{N}}&:=\Xi^w(\mu_N)&&\text{ with }\mu_N\in Z^0\Hom(N\otimes A,N),\\
\xi_{\underline{N};\underline{C}\underline{N}}&:=-\Xi^w(\delta_N)&&\text{ with }\delta_N\in Z^0\Hom(N,C\otimes N),\\
\xi_{\underline{N},\underline{M};\underline{C}}&:=\Xi^w(\mu_{N,M})&&\text{ with }\mu_{M,N}\in Z^0\Hom(N\otimes M,C),\\
\xi_{\underline{A};\underline{M},\underline{N}}&:=\Xi^w(\delta_{M,N})&&\text{ with }\delta_{N,M}\in Z^0\Hom(A,M\otimes N).
\end{aligned}
\end{equation}
\subsection{The case \boldmath $N=0$}\label{parstructure}
The Maurer Cartan equation \eqref{eq:mc2} translates into a number of easily guessed quadratic relations  between the operations
exhibited in \eqref{eq:operations}.
We list these in the case $N = 0$ (writing $\delta=\delta_M$, $\mu=\mu_M$, and observing that of course $\mu_{N,M}=0$, $\delta_{M,N}=0$, $\mu_N=0$, $\delta_N=0$).
\begin{proposition}
Maurer Cartan equation is equivalent to the following relations:
\begin{itemize}
\item[Assoc($m$):] $m(m \otimes \id) = m(\id \otimes m)$
\item[Assoc($m, \mu$):] $\mu(m \otimes \id) = \mu(\id \otimes \mu)$
\item[Coassoc($\Delta$):] $(\id \otimes \Delta)\Delta = (\Delta\otimes \id)\Delta$
\item[Coassoc($\Delta, \delta$):] $(\id \otimes \Delta)\delta = (\delta \otimes \id)\delta$
\item[Comp($\mu, \delta$):] $(\mu \otimes \id)(\id \otimes \delta) = \delta \mu$
\end{itemize}
\end{proposition}
\begin{proof} This is a somewhat tedious, but straightforward verification, similar to Example \ref{ex:assoc}.
\end{proof}
Taken together, these relations express that $A$ is a an algebra in
$\B(\alpha, \alpha)$, $C$ is a coalgebra in
$\B(\gamma, \gamma)$ and $M$ is a an object in
in $\B(\alpha, \gamma)$, with a left action $\mu$ of $A$ and a right
coaction $\delta$ of $C$ satisfying the natural compatibility
condition Comp($\mu, \delta$).
\subsection{The anatomy of the semi-co-Hochschild complex}
\label{sec:anatomy}
Let $\xi$ be a semi-coalgebra structure on $(A,M,C)$. Below we look in more detail at the structure of $\Sigma\CC_{\scHoch,\xi}(A,M,C)$ as a complex. For simplicity we will assume that
the categories $\B(\ast,\ast)$ are closed under suspensions and desuspensions so that in particular Example \ref{ex:weighted_end} applies to the properad $\uEnd$.
Note that, as explained in \S\ref{sec:notconv}, this is not a serious restriction.

We have
\begin{equation}
\label{eq:decomposition}
\Sigma \CC_{\scHoch}(A,M,C)=\Sigma \CC_{\Hoch}(A)\oplus \Sigma \CC_{\scHoch}(M)\oplus \Sigma\CC_{\coHoch}(C)
\end{equation}
where
\begin{align*}
\Sigma\CC_{\Hoch}(A)&=\prod_{m\ge 0} \Sigma \CC^m_{\Hoch}(A)&&=\prod_{m\ge 0} \Hom((\Sigma A)^{\otimes m},\Sigma A)\\
\Sigma\CC_{\coHoch}(C)&=\prod_{n\ge 0}\Sigma\CC^n_{\coHoch}(C)&&=\prod_{n\ge 0} \Hom(\Sigma^{-1} C,(\Sigma^{-1} C)^{\otimes n})
\\
\Sigma\CC_{\scHoch}(M)&=\,\prod_{\text{\hskip -1cm\rlap{$m,n\ge 0$}}} \Sigma\CC^{m,n}_{\scHoch}(M)&&=\, 
\prod_{\text{\hskip -1cm\rlap{$m,n\ge 0$}}} \Hom((\Sigma A)^{\otimes m}\otimes M,M\otimes (\Sigma^{-1} C)^{\otimes n})
\end{align*}
The differential obtained from the semi-coalgebra structure decomposes as in the following diagram
\begin{equation}
\label{eq:anatomydiff}
\xymatrix{%
\Sigma\CC_{\Hoch}(A)
\ar@(dl,dr)_{Q^1+\partial_A}
\ar[rr]^{f_{A,M}}
&&\Sigma \CC_{\scHoch}(M)\ar@(dl,dr)_{%
Q^1+\partial_{A,M}+\partial_{M,C}
}
&&\Sigma\CC_{\coHoch}(C)\ar[ll]_{f_{M,C}}
\ar@(dl,dr)_{Q^1+\partial_C}
}
\end{equation}
with
\begin{equation}
\label{eq:components}
\begin{aligned}
\partial^m_A&=[\xi_{\underline{A},\underline{A};\underline{A}},-]:&\Sigma\CC^m_{\Hoch}(A)&\r \Sigma\CC^{m+1}_{\Hoch}(A)\\
\partial^n_C&=[\xi_{\underline{C};\underline{C},\underline{C}},-]:&\Sigma\CC^n_{\coHoch}(C)&\r \Sigma\CC^{n+1}_{\coHoch}(C)\\
\partial^{m,n}_{A,M}&=[\xi_{\underline{A},\underline{M};\underline{M}},-]\pm -\bullet \xi_{\underline{A},\underline{A};\underline{A}}:
& \Sigma\CC^{m,n}_{\scHoch}(M)&\r \Sigma\CC_{\scHoch}^{m+1,n}(M)
\\
\partial_{M,C}^{m,n}&=[\xi_{\underline{M};\underline{M},\underline{C}},-]+\xi_{\underline{C};\underline{C},\underline{C}}\bullet -:
& \Sigma\CC^{m,n}_{\scHoch}(M)&\r \Sigma \CC^{m,n+1}_{\scHoch}(M)
\\
f_{A,M}^m&=\xi_{\underline{A},\underline{M};\underline{M}}\bullet-:& \Sigma\CC^{m}_{\Hoch}(A)&\r \Sigma\CC^{m,0}_{\scHoch}(M)\\
f_{M,C}^n&=\pm - \bullet \xi_{\underline{M},\underline{C};\underline{M}}:&\Sigma\CC^{n}_{\coHoch}(C)&\r \Sigma\CC^{0,n}_{\scHoch}(M)
\end{aligned}
\end{equation}
All maps in \eqref{eq:components} anti-commute with $Q^1$. Moreover we have
\begin{align*}
f_{A,M}\circ \partial_A+\partial_{A,M}\circ f_{A,M}&=0\\
f_{M,C}\circ\partial_C+\partial_{M,C}\circ f_{M,C}&=0\\
\partial_{M,C}\circ f_{A,M}&=0\\
\partial_{A,M}\circ f_{M,C}&=0\\
\partial_{A,M}\circ \partial_{M,C}+\partial_{M,C}\circ \partial_{A,M}&=0
\end{align*}
In particular we obtain a
homotopy bicartesian square
of complexes of $k$-modules
\begin{equation}
\label{lem:bicartesian}
\xymatrix{ {\Sigma\CC_{\scHoch,\xi}(A,M,C)} \ar[r]^-{\pi_C} \ar[d]_{\pi_A}
& \Sigma\CC_{\coHoch,\xi_C}(C)\ar[d]^{f_{M,C}}\\
\Sigma \CC_{\Hoch,\xi_A}(A) \ar[r]_{f_{A,M}} & \Sigma \CC_{\scHoch,\xi}(M)
}
\end{equation}
where $\xi_A=\xi_{\underline{A},\underline{A};\underline{A}}$, $\xi_C=\xi_{\underline{C};\underline{C},\underline{C}}$ and the notations are conform with Remark \ref{rem:twist}.
In particular $\Sigma \CC_{\scHoch,\xi}(M)$ is the graded $k$-module $\Sigma \CC_{\scHoch}(M)$ equipped with the differential $Q^1+\partial_{A,M}+\partial_{M,C}$.
\begin{lemma}
\label{lem:piApiC}
\begin{itemize}
\item $\pi_A$ is a quasi-isomorphism if and only if $f_{M,C}$ is a quasi-isomorphism.
\item $\pi_C$ is a quasi-isomorphism if and only if $f_{A,M}$ is a quasi-isomorphism.
\end{itemize}
\end{lemma}
Using the appropriate spectral sequences we may attempt to ``decouple'' as much as possible the $A$-action
and the $C$-coaction on $M$.
\begin{lemma}
\label{lem:coarse}
\begin{enumerate}
\item 
Assume that for every $n$ we have that
\begin{equation}
\label{eq:leftKeller}
(\Sigma \CC^n_{\coHoch}(C),Q^1)\xrightarrow{f^n_{M,C}} (\Sigma \CC^{\bullet,n} (M),Q^1+\partial^{\bullet,n}_{A,M})
\end{equation}
is a quasi-isomorphism. Then $f_{M,C}$ is a quasi-isomorphism.
\item
Assume that for every $m$ we have that
\begin{equation}
\label{eq:rightKeller}
(\Sigma \CC^m_{\Hoch}(A),Q^1)\xrightarrow{f^m_{A,M}} (\Sigma \CC^{m,\bullet} (M),Q^1+\partial^{m,\bullet}_{M,C})
\end{equation}
is a quasi-isomorphism. Then $f_{A,M}$ is a quasi-isomorphism.
\end{enumerate}
\end{lemma}
Note that the right hand side  \eqref{eq:leftKeller} only depends on the $A$-action on
$M$ and \eqref{eq:rightKeller} only depends on the $C$-coaction on $M$. 
We will refer to (1) informally as the \emph{left Keller condition} on $(A,M,C)$ and to (2) as the \emph{right Keller condition}
on $(A,M,C)$.

\section{Variations on the bar-complex}\label{sec:varbar}
We recall some well-known results concerning the bar complex. Throughout $A = (A, m, 1)$ a unital $k$-algebra.
\subsection{Generalities on coalgebras}\label{parcoalg}

Let $(\Mod(k), \otimes, k)$ be the monoidal category of $k$-modules and $(\Bimod(A), \otimes_A, A)$ the monoidal category of $A$-bimodules.\footnote{Bimodules are always assumed to be $k$-central.} Consider the forgetful functor 
$$\Bimod(A) \lra \Mod(k): X \longmapsto \bar{X}$$
and its left adjoint
$$L: \Mod(k) \lra \Bimod(A): M \longmapsto A \otimes M \otimes A.$$

The functor $L$ is oplax monoidal. For the units $k \in \Mod(k)$ and $A \in \Bimod(A)$, the natural comparison map is given by $m: L(k) = A \otimes A \lra A$, and we have the natural comparison map
$$\phi: L(- \otimes_k -) \lra L(-) \otimes_A L(-)$$ with
$$\phi_{M,N}: A \otimes M \otimes N \otimes A \lra (A \otimes M \otimes A) \otimes_A (A \otimes N \otimes A)$$
given by
$$\phi_{M,N}(a \otimes m \otimes n \otimes a') = (a \otimes m \otimes 1) \otimes (1 \otimes n \otimes a').$$
Consequently, $L$ turns $k$-coalgebras into coalgebras in $\Bimod(A)$.
Similar considerations apply to graded coalgebras and dg-coalgebras.

\medskip

Let $(C, \Delta, \epsilon)$ be a coalgebra in
$\Mod(k)$. Then  the corresponding coalgebra
$(L(C), \Delta, \epsilon)$ in $\Bimod(A)$ is given by
\begin{align*}
\Delta(a \otimes c \otimes a')& = \sum_c (a \otimes c_{(1)} \otimes 1) \otimes_A (1 \otimes c_{(2)} \otimes a'),\\
\epsilon(a \otimes c \otimes a')& = a\epsilon(c)a'.
\end{align*}
where $a\otimes c\otimes a'\in L(C)=A\otimes C\otimes A$ and $\Delta(c)=\sum_c c_{(1)}\otimes c_{(2)}$.
\subsection{The coalgebra ${{\tilde{B}A}}$}\label{parcoalg2}
We first consider the graded bar construction $\Bar{A} \in G(\Mod(k))$
of $A$, which is the cofree coalgebra cogenerated by $\Sigma A$.
More precisely, we
put
$\Bar_nA= (\Sigma A)^{\otimes n}$ and
$\Bar A=\bigoplus_{n\ge 0} \Bar_{n}A$. The counit
$\epsilon$ is given by
$\epsilon^0 = \id_k: \Bar_0A = (\Sigma A)^{\otimes 0} = k \lra
k$,
and the comultiplication
$\Delta: \Bar A  \lra \Bar A \otimes \Bar A$ which
``separates tensors'' is defined by the components
\begin{multline*}
\Delta_{p, q}: (\Sigma A)^{\otimes n} \lra (\Sigma A)^{\otimes
  p} \otimes (\Sigma A)^{\otimes q}: sa_1 \otimes \cdots\otimes sa_n
\longmapsto (sa_1 \otimes \cdots \otimes sa_p) \\\otimes (sa_{p+1}
\otimes \cdots\otimes sa_n)
\end{multline*}
for $n, p, q \in \N$ and $p + q = n$, $p,q\ge 0$, by putting $\Delta = \sum_{p,q} \Delta_{p,q}$.

Let ${{\tilde{B}A}} = L(\Bar A) \in G(\Bimod(A))$ be the graded coalgebra obtained from the oplax monoidal functor $L$ from \S \ref{parcoalg}. 
Concretely, ${{\tilde{B}A}}$ is endowed with the counit
$${\epsilon}_0 = m: {{\tilde{B}A}}_0 = A \otimes A \lra A$$
and the comultiplication $\Delta= \sum_{p,q}\Delta_{p,q}$ given by the components 
${\Delta}_{p,q}: A \otimes (\Sigma A)^{\otimes n} \otimes A \lra A \otimes (\Sigma A)^{\otimes p} \otimes A \otimes (\Sigma A)^{\otimes q} \otimes A$ with
\[
{\Delta}_{p,q}(a_0\otimes sa_1 \otimes \cdots \otimes sa_n\otimes a_{n+1}) = a_0 \otimes sa_1 \otimes \cdots \otimes sa_{p+1} \otimes 1 \otimes sa_{p+2} \otimes \cdots \otimes sa_{n} \otimes a_{n+1}.
\]
Further, ${{\tilde{B}A}}$ is endowed with the \emph{pre-Hochschild differential}
$d_{{{\tilde{B}A}}}: A \otimes (\Sigma A)^{\otimes n} \otimes A \lra A \otimes (\Sigma A)^{\otimes n-1} \otimes A$
\[
d_{{{\tilde{B}A}}}=b'\otimes\Id^{\otimes n}+\Id^{\otimes n} \otimes b''
 + \sum_{p+q+2=n}\Id^{\otimes p+1}\otimes b\otimes  \Id^{\otimes q+1}
\]
where $b=b_2$ with $b_2$  as in Example \ref{ex:assoc}, $b'=m\circ (\Id\otimes \eta^{-1})$, $b''=-m\circ(\eta^{-1}\otimes \Id)$.
The morphism ${\Delta}: {{\tilde{B}A}} \lra {{\tilde{B}A}} \otimes_A {{\tilde{B}A}}$ is seen to be a
morphism of complexes for the differential $d_{{{\tilde{B}A}}}$, so ${{\tilde{B}A}}$ is a dg-coalgebra in $\Bimod(A)$. The counit ${\epsilon}: {{\tilde{B}A}} \lra A$ is a
morphism of dg-coalgebras for general reasons. The following results are well-known.
\begin{lemma} \label{lem:counitquasi}
The counit $\epsilon:{{\tilde{B}A}}\r A$  is a quasi-isomorphism.
\end{lemma}
Note that this lemma depends on our hypothesis that $A$ is unital.
\begin{lemma} \label{lem:dualbar}
We have an isomorphism of complexes
\[
\Hom_{A-A}({{\tilde{B}A}},\Sigma A)
\cong \Sigma \CC_{\Hoch}(A)
\]
where ${{\tilde{B}A}}$ is equipped with the pre-Hochschild differential $d_{{{\tilde{B}A}}}$ and $\Sigma \CC_{\Hom}(A)$ is equipped with the differential $d_{\Hoch}=[b,-]$ where $b=\Xi(m)=\eta\circ m\circ (\eta^{-1}\otimes \eta^{-1})$.
\end{lemma}
\begin{proof}
We have
\begin{align*}
\Hom_{A-A}({{\tilde{B}A}},\Sigma A)
&=\prod_{m\ge 0} \Hom_{A-A}(A\otimes (\Sigma A)^{\otimes m} \otimes A,\Sigma A)\\
&=\prod_{m\ge 0} \Hom_{k}((\Sigma A)^{\otimes m} ,\Sigma A)\\
&=\Sigma \CC_{\Hoch}(A)
\end{align*}
The differential $\Hom_{A-A}(A\otimes (\Sigma A)^{\otimes n}\otimes A,\Sigma A)\r \Hom_{A-A}(A\otimes (\Sigma A)^{\otimes n+1}\otimes A,\Sigma A)$
 is given by
\[
d(f)=(-1)^{n-1} f\circ \left(b'\otimes\Id^{\otimes n+1}+\Id^{\otimes n+1} \otimes b'' + \sum_{p+q+2=n+1}\Id^{\otimes p+1}\otimes b\otimes  \Id^{\otimes q+1}\right)
\]
since the degree of  $f$ is $n-1$.
For $g\in \Hom_k(A\otimes (\Sigma A)^{\otimes n}\otimes A,\Sigma A)$ let $\tilde{g}(-)=g(1\otimes -\otimes 1)$ be the corresponding element of  $\Hom_{A-A}((\Sigma A)^{\otimes n},\Sigma A)$. It is then easy to see that
\[
\tilde{g}\bullet b=(g\circ (\sum_{p+q+2=n+1}\Id^{\otimes p+1}\otimes b\otimes  \Id^{\otimes q+1}))\,\tilde{}
\]
Moreover one also checks that
\[
b\bullet \tilde{g}=(-1)^{n-1} (g\circ (b'\otimes\Id^{\otimes n+1}+\Id^{\otimes n+1} \otimes b''
))\,\tilde{}
\]
It now follows
\[
d_{\Hoch}(\tilde{g})=b\bullet \tilde{g}-(-1)^{n-1} \tilde{g}\bullet b=d(g)
\]
which proves what we want.
\end{proof}

\section{Application of the semi-co-Hochschild complex}\label{sec:applic}
\subsection{Main result}
Throughout in this section $A = (A, m, 1)$ will be a unital $k$-algebra. We will use a particular semi-co-Hochschild complex to prove the following result.
\begin{theorem}[See \S\ref{sec:BA} below] \label{th:mainth2} Assume that $C$ is a counital dg-coalgebra in $C(\Bimod(A))$ such that $\epsilon:C\r A$ is a quasi-isomorphism.  Assume furthermore that~$A$ is $k$-projective and that $C$ is homotopy projective as complex of bimodules.  There is an isomorphism between $\CC_{\Hoch}(A)$ and $\CC_{\coHoch}(C)$ in the homotopy category of $B_\infty$-algebras.
\end{theorem}\label{cor:BA}
Putting $C={{\tilde{B}A}}$ as is \S\ref{parcoalg2} yields immediately:
\begin{corollary}
Assume that $A$ is $k$-projective. There is an isomorphism between $\CC_{\Hoch}(A)$ and $\CC_{\coHoch}({{\tilde{B}A}})$ in the homotopy category of $B_\infty$-algebras, where ${{\tilde{B}A}}$
is considered as a coalgebra in $C(\Bimod(A))$ as in \S\ref{parcoalg2}.
\end{corollary}
\subsection{A special semi-co-Hochschild complex}
Let  $R$, $S$ be $k$-algebras.
To $R$, $S$ we may naturally 
associate a bicategory $\B$ enriched in $C(k)$ with two objects $\alpha,\gamma$.
We put
\[
\begin{array}{ll}
\B(\alpha, \alpha) = C(\Bimod(R)) &   \B(\alpha, \gamma) = C(\Bimod(R,S))\\
 \B(\gamma, \alpha) = C(\Bimod(S,R)) &   \B(\gamma, \alpha) = C(\Bimod(S))
\end{array}
\]
 Thus the  $1$-cells are bimodules
  and the $2$-cells are morphisms of bimodules. The composition on the
  level of $1$-cells is given by the tensor product of
  bimodules. 

\medskip

We now specialise to the case $R=k$, $S=A$. Thus
\[
\begin{array}{cl}
\B(\alpha,\alpha)=C(k) & \B(\alpha,\gamma)=C(A)\\
\B(\gamma,\alpha)=C(A) & \B(\gamma,\gamma)=C(\Bimod(A))
\end{array}
\]
We consider $A$ as an  object of $C(k)=\B(\alpha,\alpha)$ (concentrated in degree zero, and with zero differential) and in addition we assume we are given $M\in C(A)=\B(\alpha,\gamma)$, $C\in C(\Bimod(A,A))=\B(\gamma,\gamma)$
as well as a semi-coalgebra structure $\xi$ on $(A,M,C)$ as in Definition \ref{def:semico} (with $N=0$). The semi-coalgebra structure will always be the same so we omit it from the notations.
E.g. we write $\CC_{\scHoch}(A,M,C)$ instead of $\CC_{\scHoch,\xi}(A,M,C)$.

We give a suitable representation of $\CC_{\scHoch}(A, M, C)$ (see \eqref{def:schoch}) based upon Hom complexes in the category $\Bimod(A)$.
More precisely we have.
\begin{lemma} \label{lem:rewrite}
The decomposition of $\Sigma \CC_{\scHoch}(A,M,C)$ in \eqref{eq:decomposition} may be written as
\begin{multline*}
\Sigma \CC_{\scHoch}(A,M,C)=\Hom_{A-A} ({{\tilde{B}A}},\Sigma A)\oplus\\
\prod_{n\ge 0}\Hom_{A-A}({{\tilde{B}A}}\otimes_A M, M\otimes_A (\Sigma^{-1} C)^{\otimes_A
  n})
\oplus \prod_{n\ge 0} \Hom_{A-A}(\Sigma^{-1} C,(\Sigma^{-1} C)^{\otimes_A n})
\end{multline*}
Under this identification we have the following correspondences
\begin{align*}
\partial_A&\leftrightarrow d_{{{\tilde{B}A}}}^\vee\\
\partial_{A,M}&\leftrightarrow (d_{{{\tilde{B}A}}}\otimes \Id_M)^{\vee}\\
f_{A,M}&\leftrightarrow f\mapsto -\mu\circ (\eta^{-1}\otimes \Id_M)\circ (f\otimes \Id_M) \in \Hom_{A-A}({{\tilde{B}A}}\otimes_A M, M)
\end{align*}
\end{lemma}
\begin{proof} The claim for $\partial_A$ is Lemma \ref{lem:dualbar}. The claims for $\partial_{A,M}$ and $f_{A,M}$ are proved
in a similar way. For $f_{A,M}$ we use in addition that $\xi_{A,M;M}=-\mu\circ(\eta^{-1}\otimes \Id_M)$ by \eqref{eq:operations}.
\end{proof}

\subsection{Some conditions}
\label{sec:conditions}
\begin{definition}\label{def:Kconditions}
Consider a complex $M \in C(\Bimod(A))$.
\begin{enumerate}
\item $M$ satisfies the \emph{left Keller condition} with respect to $X, Y \in C(\Bimod(A))$ if the morphism $$\id_M \otimes_A -: \Hom_{A-A}(X,Y) \lra \Hom_{A-A}(M \otimes_A X, M \otimes_A Y)$$
is a quasi-isomorphism.
\item $M$ satisfies the \emph{right Keller condition} with respect to $X, Y \in C(\Bimod(A))$ if the morphism $$- \otimes_A \id_M: \Hom_{A-A}(X,Y) \lra \Hom_{A-A}(X \otimes_A M, Y \otimes_A M)$$
is a quasi-isomorphism.
\end{enumerate}
\end{definition}
Below we will consider the following general conditions.
\begin{itemize}
\item[(i)] $A$ is $k$-projective and $M$ and $C$ are homotopy projective in $C(\Bimod(A))$;
\item[(ii)] $C$ is counital and $\epsilon_C: C \lra A$ is a quasi-isomorphism;
\end{itemize}
and two conditions 
\begin{itemize}
\item[(iii)] $M$ satisfies the left Keller condition with respect to $C$ and $C^{\otimes_A n}$ for $n\ge 0$.
\item[(iv)] $M$ satisfies the right Keller condition with respect to $A^{\otimes m}$ and $A$ for $m\ge 2$.
\end{itemize}
which will turn out to be related to the conditions (1)(2) in Lemma \ref{lem:coarse}. Under these conditions we will
prove the following result.
\begin{theorem} \label{th:conditions}
Assume (i)(ii)(iii)(iv) hold. Then $\pi_A$, $\pi_C$ are quasi-isomorphisms and in particular
using Proposition \ref{prop:piApiC} and Lemma \ref{lem:piApiC},
$\CC_{\Hoch}(A)$ and $\CC_{\coHoch}(C)$ are quasi-isomorphic in the homotopy category of $B_\infty$-algebras.
\end{theorem}
The proof of the fact that $\pi_A$, $\pi_C$ are quasi-isomorphisms will be
  carried out in sections \ref{sec:piA},\ref{parpiC} below.
\subsection{The projection $\pi_A$}
\label{sec:piA}
The aim in this section is to show that under suitable conditions the projection $\pi_A$ is a quasi-isomorphism. By Lemmas \ref{lem:coarse} 
this is equivalent to the left Keller condition on $(A,M,C)$, i.e.\
taking into account Lemma \ref{lem:rewrite}, to the conclusion of the following proposition
\begin{proposition}
If the conditions (i), (ii), (iii) hold, then the morphism $f^n_{M,C}$ 
\[
\Hom_{A-A}(\Sigma^{-1}C,(\Sigma^{-1}C)^{\otimes_A n})\xrightarrow{f^n_{M,C}} \Hom_{A-A}({{\tilde{B}A}}\otimes_A M, M\otimes_A (\Sigma^{-1} C)^{\otimes_A  n})
\]
is a quasi-isomorphism. 
\end{proposition}
\begin{proof}
Taking into account Lemma \ref{lem:rewrite},
up to sign $f^n_{M,C}$ is given by the following composition (since we do not care about the signs, we do not
write the shifts)
\begin{equation}\label{firstcomp}
\xymatrix{ {\Hom_{A-A}(C, C^{\otimes_A n})} \ar[d]^{\Id_M \otimes_A -} \\ 
{\Hom_{A-A}(M \otimes_A C, M \otimes_A C^{\otimes_S n})} \ar[d]^{-\circ\delta} \\
{\Hom_{A-A}(M, M \otimes_A C^{\otimes_A n})} \ar[d]^{-\circ \mu} \\
{\Hom_{A-A}(A \otimes_A M, M \otimes_A C^{\otimes_A n})}  \ar[d]^{- \circ (\epsilon_{{{\tilde{B}A}}} \otimes_A \id_M)}\\
{\Hom_{A-A}({{\tilde{B}A}} \otimes_A M, M \otimes_A C^{\otimes_A n})} }
\end{equation}
  We analyse the 4 individual components of the decomposition in
  \eqref{firstcomp} separately.  The first morphism $\id_M \otimes_A -$
  is a quasi-isomorphism from (iii). All three other morphisms take
  the form $-\circ g$ for a morphism $g$ between homotopy projective
  $A$-bimodules, where we have used (i) (in particular the $k$-projectivity of $A$ is used to guarantee that ${{\tilde{B}A}}$ is homotopy projective in $C(\Bimod(A))$). Hence, in each case, it is
  sufficient that the morphism $g$ is a quasi-isomorphism in order for
  $-\circ g$ to be a quasi-isomorphism. For $g = \delta$, this follows from
  (ii) and Lemma \ref{lemdeltaqis} below. For $g = \mu$, this follows from
  the fact that $\mu$ is an isomorphism of complexes. For
  $g = \epsilon_{{{\tilde{B}A}}} \otimes_A \id_M$, this follows from the fact that
  $\epsilon_{{{\tilde{B}A}}}$ is a quasi-isomorphism and $M$ is homotopy
  projective.
\end{proof}
We have used:
\begin{lemma}\label{lemdeltaqis}
Assume (i)(ii) hold.
Then for every co-unital right $C$-comodule $M$, the co-action $\delta: M \lra M \otimes_A C$ is a quasi-isomorphism.
\end{lemma}
\begin{proof}
Since $A$ is a left projective $A$-module by (i) the map $\epsilon_C$ is split as left $A$-modules. Hence $M\otimes_A C\xrightarrow{\Id_M \otimes \epsilon_{C}} M\otimes_A A\cong M$ is a quasi-isomorphism
(using for example that by (ii) $M$ is right homotopy projective). It now suffices to note that the composition
$
M\xrightarrow{\delta} M\otimes_A C\xrightarrow{\Id_M\otimes \epsilon} M
$
is the identity.
\end{proof}

\subsection{The projection $\pi_C$}\label{parpiC}
The aim in this section is to show that under suitable conditions the projection $\pi_C$ is also a quasi-isomorphism.
Lemma \ref{lem:coarse}  this is equivalent to the right Keller condition on $(A,M,C)$. This is covered in the next proposition.
\begin{proposition}
\label{propalt}
If the conditions (i), (ii), (iv) hold then for every $m\ge 0$
\begin{multline*}
\Hom_{A-A}(({{\tilde{B}A}})_m,A)\\\xrightarrow{f_{A,M}}\left(\prod_{n\ge 0}\Hom_{A-A}(({{\tilde{B}A}})_m\otimes_A M,M\otimes_A (\Sigma^{-1}C)^{\otimes_A n}),d=Q^1+\partial_{M,C}\right)
\end{multline*}
is a quasi-isomorphism.
\end{proposition}
\begin{proof}
To simplify the notations we write $X=({{\tilde{B}A}})_m$.
Using condition (iv) it is sufficient to prove that the projection
\begin{multline*}
\left(\prod_{n\ge 0}\Hom_{A-A}({X}\otimes_A M,M\otimes_A (\Sigma^{-1}C)^{\otimes_A n}),d=Q^1+\partial_{M,C}\right)\\
\lra 
\biggl(\Hom_{A-A}({X}\otimes_A M,M),d=Q^1\biggr)
\end{multline*}
is a quasi-isomorphism. 
By (ii) $\epsilon_C$ is a quasi-isomorphism. Hence by (i), 
\[
(\Sigma^{-1} C)^{\otimes_A n}\r (\Sigma^{-1}A)^{\otimes_A n}
\]
is a quasi-isomorphism. Using (i) again it follows that we have to prove that the projection
\begin{multline}
\label{eq:projection}
\left(\prod_{n\ge 0}\Hom_{A-A}({X}\otimes_A M,M\otimes_A (\Sigma^{-1}A)^{\otimes_A n}),d=Q^1+\partial_{M,A}\right)\\
\lra \biggl(\Hom_{A-A}({X}\otimes_A M,M),d=Q^1\biggr)
\end{multline}
is a quasi-isomorphism where $\partial_{M,A}$ is deduced from the trivial semi-coalgebra structure $(A,M,A)$ where all the data in the ``co-part'' (represented by $\delta:M\r M\otimes_A A$, $\Delta:A\r A\otimes_A A$, $\epsilon:A\r A$) consists of identity morphisms.

Let $\sigma_n$ be as in the following commutative diagram
\[
\xymatrix@C=1.5em{
\Hom_{A-A}(X\otimes_A M,M\otimes_A (\Sigma^{-1}A)^{\otimes_A n})\ar[d]_{\partial_{M,A}}\ar[r]^-{\cong}&\Hom_{A-A}(X\otimes_A M,\Sigma^{-n}M)\ar[d]^{\sigma_n}\\
\Hom_{A-A}(X\otimes_A M,M\otimes_A (\Sigma^{-1}A)^{\otimes_A (n+1}))\ar[r]_-{\cong}&\Hom_{A-A}(X\otimes_A M,\Sigma^{-n-1}M)
}
\]
One easily verifies
\[
\sigma_n=
\begin{cases}
\pm \eta^{-1}\circ -&\text{if $n$ is odd}\\
0&\text{if $n$ is even}
\end{cases}
\]
It follows that the left hand side of \eqref{eq:projection} (viewed as a double complex) really looks like
\begin{multline*}
\biggl(\Hom_{A-A}({X}\otimes_A M,M),d=Q^1\biggr)\xrightarrow{0}
\biggl(\Hom_{A-A}({X}\otimes_A M,M),d=Q^1\biggr)\\\xrightarrow{\cong}
\biggl(\Hom_{A-A}({X}\otimes_A M,M),d=Q^1\biggr)\xrightarrow{0}
\cdots
\end{multline*}
Hence the projection on the first column is a quasi-isomorphism.
\end{proof}

\subsection{Proof of Theorem \ref{th:mainth}}
\label{sec:BA}

In this section, we give the proof of Theorem \ref{th:mainth}.
For $A$, $C$ as in the statement of the theorem
we put $M=C$ and we endow
$C$ with the
canonical structure of $A$-$C$-bimodule. According to \S
\ref{parstructure}, these data define a semi-coalgebra structure on
$(A, M, C)$.

In view of Theorem \ref{th:conditions} it is sufficient to prove that the conditions (i)(ii)(iii)(iv) from
\S\ref{sec:conditions} hold. (i) is clear. (ii) is Lemma \ref{lem:counitquasi}. Finally (iii)(iv) follow
from Lemma \ref{lemlem} below.

\begin{lemma}\label{lemlem}
  Consider $X, Y$ complexes of $A$-bimodules with $X$
  homotopy-projective.
Then $M = C$ satisfies the left and
  right Keller condition with respect to $X$ and~$Y$.
\end{lemma}

\begin{proof}
The proofs of the two statements are entirely analogous, we look at the first statement.
For $\phi \in \Hom_A(X,Y)$, we have a commutative diagram
\[
\xymatrix{ {X \otimes_A C} \ar[d]_{\phi \otimes \id_{C}} \ar[r]^-{\id_X \otimes \epsilon} & X \ar[d]^{\phi} \\ {Y \otimes_A C} \ar[r]_-{\id_Y \otimes \epsilon} & Y }
\]
with the rows being quasi-isomorphisms (as $\cone (C\r A)$ is homotopy projective complex on the left and right and is therefore contractible on the left and right).

We may thus represent $- \otimes_A C$ in the derived category as a zigzag of quasi-isomorphism:
\[
\xymatrix{ {\Hom_A(X,Y)} \ar[r]_-{-(\id \otimes \epsilon)} & {\Hom_A(X \otimes_A C, Y)} & {\Hom_A(X \otimes_A C, Y \otimes_A C).} \ar[l]^-{(\id_Y \otimes \epsilon)-} }
\qedhere\]
\end{proof}

\bibliographystyle{amsplain}
\bibliography{BibfileBinfty1}

\end{document}